\documentclass[12pt, reqno]{amsart}%
\usepackage{amsmath, amsthm, amscd, amsfonts, amssymb, graphicx, color}
\usepackage[bookmarksnumbered, colorlinks, plainpages]{hyperref}
\usepackage{amsmath}
\usepackage{amsfonts}
\usepackage{amssymb}
\usepackage{graphicx}%
\providecommand{\U}[1]{\protect\rule{.1in}{.1in}}
\makeatletter
\@namedef{subjclassname@2020}{	\textup{2020} Mathematics Subject Classification}
\makeatother
\newtheorem{theorem}{Theorem}[section]
\newtheorem{lemma}[theorem]{Lemma}
\newtheorem{proposition}[theorem]{Proposition}
\newtheorem{corollary}[theorem]{Corollary}
\theoremstyle{definition}
\newtheorem{definition}[theorem]{Definition}

\newtheorem{question}[theorem]{Question}

\newtheorem{remark}[theorem]{Remark}
\numberwithin{equation}{section}

\newcommand{\diam}{\mathop{\mathrm{diam}}}
\newcommand{\be}{\begin{equation}}
\newcommand{\ee}{\end{equation}}

\begin{document}
\title[On Difference Graph]{On the Difference of the Enhanced Power Graph and the Power Graph of a finite Group}

\author[Biswas]{Sucharita Biswas}
\email{sucharita.rs@presiuniv.ac.in }
\address{Department of Mathematics\\ Presidency University, Kolkata, India} 

\author[Cameron]{Peter J. Cameron}
\email{pjc20@st-andrews.ac.uk}
\address{School of Mathematics and Statistics\\ University of St. Andrews, U.K} 

\author[Das]{Angsuman Das$^{\flat}$}
\email{angsuman.maths@presiuniv.ac.in}
\address{Department of Mathematics\\ Presidency University, Kolkata, India} 

\author[Dey]{Hiranya Kishore Dey}
\email{hiranya.dey@gmail.com}
\address{Department of Mathematics\\ Indian Institute of Science, Bangalore, India} 
\thanks{$^\flat$Corresponding author}

\begin{abstract}
 The difference graph $D(G)$ of a finite group $G$ is the difference of enhanced power graph of $G$ and power graph of $G$, with all isolated vertices are removed. In this paper we study the connectedness and perfectness of $D(G)$ with respect to various properties of the underlying group $G$. We also find several connection between the difference graph of $G$ and the Gruenberg–Kegel graph of $G$.
\end{abstract}
\maketitle
\setcounter{page}{1}
\section{Introduction}

The study of graphs related to various algebraic structures has been a topic of increasing interest during the last two decades. This kind of study help us to  (1) characterize the resulting graphs, (2) 
characterize the algebraic structures with isomorphic graphs, and (3) also to realize 
the interdependence between the algebraic structures and the corresponding graphs. Many different types of graphs, including among many others the
commuting graph~\cite{bf}, generating graph~\cite{gk},
power graph~\cite{chakraborty-ghosh-sen-power,kscc}, enhanced power graph 
\cite{akbari-cameron,sudip-hiranya,bera-hiranya-mukherjee-power}, and
comaximal subgroup graph \cite{das-saha-comaximal}, have been introduced to explore the properties of algebraic structures using graph theory.
The concept of a power graph was introduced in the context of semigroup theory by Kelarev and Quinn~\cite{Kelarev-Quin-cotrib-general}. 

\begin{definition}
Let $G$ be a group. The power graph $\mathsf{Pow}(G)$ is an undirected graph defined on $G$ as the set of vertices, in which  two vertices $a$ and $b$ are adjacent if $a$ is a power of $b$ or $b$ is a power of $a$, i.e., $a\in \langle b\rangle$ or $b\in \langle a\rangle$. 
\end{definition}

The enhanced power graph of a group was introduced by Alipour et al. in \cite{akbari-cameron} as follows.  

\begin{definition}
Let $G$ be a group. The enhanced power graph $\mathsf{EPow}(G)$ is an undirected graph defined on $G$ as the set of vertices and two vertices $a$ and $b$ are adjacent if there exists $c \in G$ such that both $a$ and $b$ are powers of $c$, i.e., $a,b \in \langle c\rangle$, i.e., if $\langle a,b \rangle$ is a cyclic group.
\end{definition}

Recently, in a survey, Cameron \cite{survey} introduced various open questions on graphs defined on groups. One of them is regarding the difference of enhanced power graph and power graph of a group.

From Proposition 2.6 in \cite{survey}, we see that both $\mathsf{Pow}(G)$ and $\mathsf{EPow}(G)$ are graphs on same vertex set $G$ and $E(\mathsf{Pow}(G))\subseteq E(\mathsf{EPow}(G))$. $\mathsf{EPow}(G)-\mathsf{Pow}(G)$ denotes the graph with $G$ as the set of vertices and two vertices $a$ and $b$ are adjacent if they are adjacent in $\mathsf{EPow}(G)$ but not adjacent in $\mathsf{Pow}(G)$. Motivated by Section 3.2 of \cite{survey}, we define the following graph;

\begin{definition}
Let $G$ be a group. The difference graph $D(G)$ is defined to be the graph $\mathsf{EPow}(G)-\mathsf{Pow}(G)$, with isolated vertices removed.
\end{definition}

Given this, we need to understand the set of isolated vertices which are
removed, that is, vertices which have the same neighbourhoods in 
$\mathsf{Pow}(G)$ and $\mathsf{EPow}(G)$. This is done in the next two
sections, where we also note a connection between $D(G)$ and the
Gruenberg--Kegel graph of $G$. (This graph, sometimes called the prime graph,
is connected with several graphs defined on $G$, as discussed in the
survey~\cite{cm}.)

The next main result of this work shows the universality of the difference graph $D(G)$. Theorem \ref{universal-theorem} shows that given any finite graph $\Gamma$, there exists a finite abelian group $G$ such that $\Gamma$ is an induced subgraph of $D(G)$. 

We next concentrate on the connectedness of the difference graph. There have been many works on connectivity of various graphs in recent times. 
The authors Aalipour et al. in \cite[Question 40]{akbari-cameron} asked about the connectivity of
power graphs when all the dominating vertices are removed. Later, Cameron and
Jafari \cite{cameron-jafari-power}  answered this question for power graphs and Bera {\it et.al.} \cite{bera-hiranya-mukherjee-power} answered this question for enhanced power graphs. Regarding connectedness of $D(G)$, our main results are
\begin{itemize}
	\item if $G$ is a finite group which is not a $p$-group and with non-trivial center, then $D(G)$ is connected. (Theorem \ref{non-trivial-center})
	\item if $G_1,G_2, G_3$ be three finite groups such that $G_1\times G_2 \times G_3$ is not a $p$-group, then $D(G_1\times G_2 \times G_3)$ is connected. (Theorem \ref{direct-product-theorem})
	\item $D(S_n)$ is connected if and only if $n \geq 8$. (Theorem \ref{Sn-connected-theorem})
	\item If $n\geq 18$, $D(A_n)$ is connected. (Theorem \ref{An-connected-theorem})
	\item $D(D_n)$ is a connected graph if and only if $n$ is not a prime power. (Theorem \ref{Dn-connected-theorem})
	\item  $D(D_n \times D_m)$ is disconnected if and only if $n$ and $m$ are powers of same odd prime. (Theorem \ref{direct-product-dihedral})
\end{itemize}

Moving on, we are also interested in the perfectness of the difference graph $D(G)$. The motivation for studying the perfectness stems from the fact that the power graph of a finite group is always perfect (Theorem 5, \cite{power-perfect}) but the question that for which finite groups, the enhanced power graph is perfect is still unresolved, although the chromatic number of the enhanced power
graph is now known~\cite{cp}, and these graphs are weakly perfect.
 
In this context we prove the following results:
\begin{itemize}
	\item $D(\mathbb{Z}_n)$ is perfect. (Theorem \ref{cyclic-perfect}).
	\item Let $G$ be a group of order $pq,p^2q,p^2q^2,p^3q$ or $pqr$, where $p,q,r$ are distinct primes. Then $D(G)$ is perfect. (Theorem \ref{pqr-perfect})
	\item We also classify the finite nilpotent groups $G$ for which $D(G)$ is perfect. (Theorem \ref{nilpotent-perfect-theorem})
	\item We give several further examples of groups whose difference graphs
are perfect, and others whose difference graphs are imperfect, including a
number of finite simple groups.
\end{itemize} 

\section{Isolated vertices}

It is proved in Aalipour \emph{et~al.} \cite{akbari-cameron} that the power graph and
enhanced power graph of $G$ are equal if and only if every element of $G$
has prime power order. Groups with these properties are known as EPPO groups.
After pioneering work by Higman~\cite{higman:eppo} (who classified the
soluble ones in the 1950s) and Suzuki~\cite{sz2} (who classified the simple
ones in the 1960s), Brandl~\cite{brandl} gave a description of EPPO groups
in 1981. The paper was not well-known, and several authors published similar
results. An accessible account appears in the survey \cite{cm}.

EPPO groups are those groups $G$ for which $D(G)$ has no edges. The next
obvious question is: Which are the isolated vertices which are deleted in the
construction of $D(G)$? To simplify the discussion, in this section and the
next we usually consider the graph $D(G)$ \emph{before} deleting isolated vertices.

\begin{proposition}
Let $G$ be a group with order greater than~$1$. Then the non-identity
element $g$ is an isolated vertex in $D(G)$ if and only if either
$\langle g\rangle$ is a maximal cyclic subgroup of $G$, or every cyclic
subgroup of $G$ containing $g$ has prime-power order.
\end{proposition}

\begin{proof}
If $\langle g\rangle$ is a maximal cyclic subgroup, and $g$ and $h$ are joined
in $D(G)$, then $\langle g,h\rangle$ is cyclic, so $h\in\langle g\rangle$,
 so $g$ and $h$ are joined in the power graph, a contradiction.

If every cyclic subgroup of $G$ containing $g$ has prime power order, and
$g$ and $h$ are joined in $D(G)$, then $\langle g,h\rangle$ has prime power
order, whence $g$ and $h$ are joined in the power graph, again a contradiction.

Conversely suppose that $\langle g\rangle$ is properly contained in a cyclic
subgroup $\langle h\rangle$ of $G$ whose order is not a prime power. Then there
is a prime $p$ such that the $p$-part of $o(h)$ is greater than that of $o(g)$.
We may also assume that, if $o(g)$ is a power of a prime $q$, then $p\ne q$.
Let $k$ be an element in $\langle h\rangle$ whose order is the $p$-part of
$o(h)$. Then $g$ and $k$ are joined in the enhanced power graph but not in
the power graph, so they are joined in $D(G)$.
\end{proof}

Since every finite group has a maximal cyclic subgroup, we see that the set of
isolated vertices which are deleted in the definition is non-empty.

\section{Connection with the Gruenberg--Kegel graph}

The \emph{Gruenberg--Kegel graph} (or GK-graph, for short) of a finite group
$G$ is the graph whose vertices are the prime divisors of $|G|$, with primes
$p$ and $q$ joined by an edge if and only if $G$ contains an element of
order~$pq$.

This was introduced by Gruenberg and Kegel in their study of the integral
group ring of a finite group, and showed that the augmentation ideal is
decomposable if and only if the GK-graph is disconnected. They proved a
structure theorem for graphs with disconnected GK-graph, but did not
publish it; a proof was given by Gruenberg's student Williams~\cite{williams},
and refined by subsequent authors. We note a connection with the preceding
section:

\begin{proposition}
The finite group $G$ is an EPPO group if and only if its GK graph has no edges.
\end{proposition}

The GK graph of a finite group has several connections with various graphs
defined on groups (some of these are listed in \cite{cm}). We can add another
one here.

\begin{theorem}
Let $G$ be a finite group whose order is divisible by the prime $p$. Then
the following are equivalent:
\begin{enumerate}
\item every element of order $p$ is an isolated vertex in $D(G)$;
\item every element of $p$-power order is an isolated vertex in $D(G)$;
\item $p$ is an isolated vertex in the GK graph of $G$;
\item the centralizer of every element of order $p$ is a $p$-group.
\end{enumerate}
\label{akbari-theorem}
\end{theorem}

\begin{proof}
If (a) fails, then some element of order $p$ is contained in a cyclic subgroup
whose order is not a power of $p$, and hence contains an element of order
$q\ne p$; so (c) and (d) also fail. If (a) is true, then no element of order
$p$ can commute with any element of order coprime to $p$, so (b), (c) and (d)
hold also. Moreover, clearly (b) implies (a). 
\end{proof}

Groups satisying (d) of the above theorem have had a lot of attention,
especially in the case $p=2$, where they are called \emph{CIT groups}
(the centralizer of an involution is a $2$-group). The simple CIT groups
were determined by Suzuki~\cite{sz1,sz2}. For odd~$p$, when such groups
are known as C$pp$ groups, Higman and his students have a number of
results, for some of which we refer to \cite{higman:odd,cm}.

\begin{theorem}
Let $G$ be a finite group which is not of prime power order, and suppose that
$D(G)$ is connected. Then the Gruenberg--Kegel graph of $G$ is connected,
apart from possibly some isolated vertices.
\end{theorem}

\begin{proof}
Let $\{g,h\}$ be an edge of $D(G)$. Then $g$ and $h$ are contained in a cyclic
group, and neither of $o(g)$ and $o(h)$ divides the other. So there is a
prime $p$ dividing $o(g)$ to a higher power than $o(h)$, and a prime $q$
dividing $o(h)$ to a higher power than $o(g)$. Put $g'=g^{o(g)/p}$ and
$h'=h^{o(h)/q}$. Then $g'$ has order $p$, $h'$ has order $q$, and
$\{p,q\}$ is an edge in the GK graph of $G$.

Now let $(g,h,k)$ be a path of length~$2$ in $D(G)$. Then there are
elements $g'$ and $h'$ of orders $p$ and $q$ as in the above paragraph,
and elements $h''$ and $k''$ of prime orders $r$ and $s$ so that
$\{h'',k''\}$ is an edge of $D(G)$. If $r=q$ then we can assume that
$h'=h''$ and we have a path of length $2$; otherwise, $\{h',h''\}$ is an
edge of $D(G)$, since $h',h''\in\langle h\rangle$. Thus we can replace
the path of length $2$ by a walk of length $2$ or $3$ all of whose vertices
have prime order; their orders are the vertices of a walk in the GK graph.

Now suppose that $D(G)$ is connected. Choose two primes $p$ and $q$ which
divide $|G|$, and take elements $x$ and $y$ of orders $p$ and $q$
respectively. We may assume that $x$ and $y$ are not isolated in $D(G)$.
By hypothesis, there is a path from $x$ to $y$, which by
the previous construction gives us a walk from $p$ to $q$ in the GK graph.
So the GK graph is connected.
\end{proof}

We note that the converse of this result is false. For example, consider the
group $G=S_3\times S_3$. The GK graph has two vertices, the primes $2$ and $3$,
joined by an edge. The vertices of $D(G)$ are the non-identity elements of the
two direct factors, and the graph consists of a complete bipartite graph on
the elements of order $2$ in the first factor and those of order~$3$ in the
second, and another complete bipartite graph where the roles of the two
factors are reversed.

\section{Preliminaries} 

Before starting the main results, we recall a few results and prove some lemmas that will be crucial in the forthcoming sections. For any element $a$ in a group $G$, $o(a)$ denotes the order of the element $a$ in the group $G$.

We begin with some observations about cyclic groups.

Suppose that $G=\mathbb{Z}_n$, the cyclic group of order~$n$. Then the enhanced power
graph of $G$ is complete, and so $D(G)$ is the complement of the power graph.
We can use this to read off some parameters of $D(G)$, using results from
\cite[Section 8.2]{kscc}.
\begin{itemize}
\item The independence number and clique cover number of $D(G)$ are equal
to the clique number and chromatic number of $\mathsf{Pow}(G)$. These numbers
are equal, since $\mathsf{Pow}(G)$ is perfect, and the common value is given
by $f(n)$, where $f$ is defined by the recursion
\[f(n)=\begin{cases}
1 & \hbox{ if $n=1$},\\
\phi(n)+f(n/p) & \hbox{if $n>1$},
\end{cases}\]
where $\phi$ is Euler's totient and $p$ is the smallest prime divisor of $n$.
It satisfies $\phi(n)\le f(n)\le c\phi(n)$, where 
\[c=\sum_{n\ge0}\prod_{i=1}^n\frac{1}{p_i-1},\]
where $p_1,p_2,\ldots$ are the primes in order; so $c=2.6481017597\dots$.
(We note that these values include the isolated vertices; the number of these
should be subtracted to get the values in the graph $D(G)$ as we have defined
it.)
\item The clique number and chromatic number of $D(G)$ are equal to the
independence number and clique cover number of $\mathsf{Pow}(G)$, which are
equal, and equal to the size of the largest antichain of divisors of $n$,
as we discuss later.
\end{itemize}

See \cite{kscc,cp} for proofs and further details.
\bigskip

\begin{proposition}
A finite group $G$ is nilpotent if and only if for all $x,y \in G$ with 
$\gcd(o(x),o(y))=1$ we have $xy=yx$.
\end{proposition}

\begin{proposition}[\cite{bera-hiranya-mukherjee-power}, Lemma 2.5] \label{coprime}
Let $G$ be a finite group and let  $a,b$ be non-identity elements of $G$
such that $ab=ba$ and $\gcd(o(a),o(b))=1$. Then $a\sim b$ in $D(G)$.
\end{proposition}

From the above two propositions, we have the following corollary
\begin{corollary} \label{nilpotent-tool}
In a finite nilpotent group $G$, if $\gcd(o(x),o(y))=1$, then $x \sim y$ in $D(G)$.
\end{corollary}

\begin{lemma} \label{subgroup-induced}
Let $H$ be a subgroup of a group $G$. Then $D(H)$ is an induced subgraph of $D(G)$.
\end{lemma}

\begin{proof}
If the statements ``$y$ is a power of $x$'' and ``$\langle x,y\rangle$ is
cyclic'' are true in $G$ then they are true in any subgroup of $G$ containing
$x$ and $y$.
\end{proof} 

\begin{proposition}\label{p-subgroup}
Let $p$ be a prime. Then the set of elements of $p$-power order in $D(G)$
contains no edges.
\end{proposition}

\begin{proof}
Suppose that $\{x,y\}$ is an edge, where both $x$ and $y$ have prime power
order. Then $x$ and $y$ are contained in a cyclic group of prime power order,
so one is a power of the other, a contradiction.
\end{proof}

\begin{lemma} \label{to-prime-power}
Let $G$ be a group and $x\in G$ be a (non-isolated) vertex in $D(G)$.
\begin{enumerate}
\item If $o(x)=p^\alpha$ for some prime $p$, then there exists a prime $q(\neq p)$ and $y \in G$ such that $o(y)=q^\beta$ and $x\sim y$ in $D(G)$.
\item If $o(x)$ is not a prime power, then there exists a prime $p$ and $y \in G$ such that $o(y)=p^\alpha$ and $x\sim y$ in $D(G)$.
\end{enumerate} 
\end{lemma}

\begin{proof}
As $D(G)$ has no isolated vertex, there exists $g \in G$ such that $x \sim g$ in $D(G)$. Thus $\langle x,g \rangle$ is cyclic. Let $\langle x,g\rangle=\langle z\rangle$.
\begin{enumerate}
\item Let $o(x)=p^\alpha$.  If $o(z)=p^{\alpha_{1}}$, then by Proposition~\ref{p-subgroup}, we have $x \not\sim g$, a contradiction. Thus there exists a prime $q(\neq p)$ such that $q|o(z)$. Let $o(z)=q^\beta m$, where $\gcd(m,q)=1$ and set $y=x^m$. Then $o(y)=q^\beta$. Thus $o(x)$ and $o(y)$ are relatively prime, and $\langle x, y \rangle$ is a subgroup of $\langle z \rangle$. Thus $x \sim y$ in $D(G)$.
\item Let $o(x)$ be such that it is not a prime power. If $o(x)=o(z)$, then $\langle x \rangle=\langle z\rangle$, i.e., $z \in \langle x \rangle$. Also, as $g\in \langle z \rangle$, we have $g\in \langle x \rangle$, i.e., $x\not\sim g$ in $D(G)$, a contradiction. Thus $o(x)<o(z)$. Let $p$ be a prime such that $p$ divides $o(z)/o(x)$. Let $o(z)=p^\alpha m$, where $\gcd(m,p)=1$ and $p^\alpha$ does not divide $o(x)$. Also, as $o(x)$ is not a prime power, $o(x)$ does not divide $p^\alpha$. 
		
Set $y=z^m$. Then $o(y)=p^\alpha$ and $\langle x,y \rangle$ is a subgroup of $\langle z \rangle$. Moreover, $o(x)$ and $o(y)$ does not divide each other, i.e., $x$ and $y$ are not powers of each other. Thus $x\sim y$ in $D(G)$. 
\end{enumerate}
\end{proof}

\begin{theorem}\label{universal-theorem}
The class of difference graphs of groups is universal: that is, given any graph $\Gamma$, there exists a finite abelian group $G$ such that $\Gamma$ is an induced subgraph of $D(G)$.
\end{theorem}

\begin{proof}
	The proof is by induction on the number of vertices of $\Gamma$. For a graph with a single vertex, it is obvious. Assume that the result holds for all graphs with $n-1$ vertices. Now let $\Gamma$ be a
	graph with vertex set $\lbrace 1,2, \cdots ,n\rbrace$, and $\Gamma'$ be its induced subgraph on the vertices $\{1,2,\ldots,n-1\}$. From induction hypothesis, let $\varphi$ be an isomorphism from $\Gamma'$ to an induced subgraph of $D(G)$, for a finite abelian group $G$. 
	
	Let $p$ be a prime not dividing the order of $G$, and let $H=\langle a, b \rangle$ be an elementary abelian group of order $p^2$. Consider the group $G \times H$ and the map $\tilde{\varphi}: \Gamma \rightarrow G\times H$ given by 
	$$\tilde{\varphi}(i)=\left\lbrace \begin{array}{cc}
	(\varphi(i),a) & \mbox{ if $i<n$ and $(i,n)\notin E(\Gamma)$}\\
	(\varphi(i),e') & \mbox{ if $i<n$ and $(i,n)\in E(\Gamma)$}\\
	(e,b) & \mbox{ if $i=n$}
	\end{array} \right.$$
	where $e,e'$ are the identity elements of the groups $G$ and $H$ respectively. 
	
	Since $p \nmid |G|$, for any $z \in G$ we have $\langle (z,e'), (e,b) \rangle = \langle (z,e') \rangle \times \langle (e,b) \rangle$, which is cyclic and not generated by either element. So the embedding of $\lbrace 1,2, \cdots , n-1 \rbrace$ is still an isomorphism to an induced subgraph. Moreover, $\langle (\varphi(i),e'),(e,b) \rangle$ is cyclic while $\langle (\varphi(i),a), (e,b) \rangle$ 
	is not, so we have the correct edges from $(e,b)$ to the other vertices, and the result is proved. 
\end{proof}

\section{When $G$ has non-trivial center}
In this section, we deal with $D(G)$ when $G$ has a non-trivial center. If $G$ is a finite $p$-group, then its every cyclic subgroup has prime power order and hence by Theorem \ref{akbari-theorem}, $\mathsf{EPow}(G)=\mathsf{Pow}(G)$, i.e., $D(G)$ is the null graph. So, we focus on non $p$-groups. 

\begin{theorem}\label{non-trivial-center}
	Let $G$ be a finite group which is not a $p$-group. If $G$ has non-trivial center, then $D(G)$ is connected and $\diam(D(G))\leq 6$.
\end{theorem}

\begin{proof}
As $|Z(G)|>1$, there exists a prime $p$ such that $p\mid |Z(G)|$. Thus there exists $z \in Z(G)$ with $o(z)=p$. We show that any vertex in $D(G)$ is joined by a path to $z$. Let $x \in V(D(G))$. Then by Lemma \ref{to-prime-power}, there exists $y \in V(D(G))$ such that $x \sim y$ and $o(y)=q^\alpha$. If $p\neq q$, we have $x \sim y \sim z$. If $p=q$, as $y\in V(D(G))$, again by Lemma \ref{to-prime-power}, there exists a prime $r \neq p$ and an element $y' \in V(D(G))$ such that $y \sim y'$ and $o(y')=r^\beta$. Thus $x\sim y \sim y' \sim z$, i.e., $d(x,z)\leq 3$. Similarly, for another vertex $x'$, $d(x',z)\leq 3$. (Note that $z$ is same in both the cases) Hence for two arbitrary vertices $x,x'$, we have $d(x,x')\leq 6$
\end{proof}

{\remark The bound in the above theorem is tight. Take $G=\mathbb{Z}_7 \rtimes \mathbb{Z}_{12} $ (GAP ID: (84,1)) \cite{GAP4}. It is a non-nilpotent, super-solvable group with $|Z(G)|=2$ and $\diam(D(G))=6$. If $G$ is a finite nilpotent group which is not a $p$-group, then $G$ has a non-trivial center, and as a result, by Theorem \ref{non-trivial-center}, $D(G)$ is connected and $\diam(D(G))\leq 6$. However, for nilpotent groups this bound can be improved to $4$.}

\begin{theorem}\label{nilpotent-connected}
	If $G$ is a finite nilpotent group which is not a $p$-group, then $D(G)$ is connected and $\diam(D(G))\leq 4$. 
\end{theorem}
\begin{proof}
	Let $\pi(G)$ denote the set of prime divisors of $|G|$. As $G$ is a not a $p$-group, $|\pi(G)|\geq 2$. Again, as $G$ is a finite nilpotent group, $G$ is the direct product of its Sylow subgroups, say $P_1,P_2,\ldots,P_k$, where $k\geq 2$. Also, $Z(G)=Z(P_1)\times Z(P_2)\times \cdots \times Z(P_k)$, i.e., $|Z(G)|$ has atleast two distinct prime factors.

	 Let $x,x' \in V(D(G))$. Then by Lemma \ref{to-prime-power}, there exists $y,y'\in V(D(G))$ such that $x\sim y, x'\sim y'$ and $o(y)=p^\alpha$, $o(y')=q^\beta$ for primes $p$ and $q$. If $p\neq q$, then by Corollary \ref{nilpotent-tool}, $y\sim y'$ and $d(x,x')\leq 3$. If $p=q$, as $|Z(G)|$ has atleast two distinct prime factors, there exists an element $z \in Z(G)$ of prime order $r(\neq p)$ such that $y \sim z \sim y'$. In this case, we have $d(x,x')\leq 4$.
\end{proof}

\begin{remark}
The bound in the above theorem is tight: Take $G=\mathbb{Z}_4\times \mathbb{Z}_4\times \mathbb{Z}_6$. Then $\diam(D(G))=4$ and $(0,2,4),(2,2,4)$ are two antipodal vertices in $D(G)$. 
\end{remark}

\begin{theorem}\label{Dn-connected-theorem}
Let $D_n$ be the dihedral group of order $2n$, then 
\begin{itemize}
	\item $D(D_n)$ is an empty graph, if $n=p^r$ for some prime $p$.
	\item $D(D_n)$ is a connected graph, if $n$ is not a prime power. 
\end{itemize}
\end{theorem}

\begin{proof}
Let $D_n=\langle a,b: a^n=b^2=1; bab=a^{-1} \rangle$. As elements of the form $a^ib$ in $D_n$ are of order $2$ and only cyclic subgroup containing $a^ib$ is $\{e,a^ib\}$, $a^ib$ are isolated vertices in $\mathsf{EPow}(D_n)$ and hence they are also isolated vertices in $D(G)$. Thus the remaining vertices in $D(D_n)$ are the elements of the cyclic group  of order $n$ generated by $a$. Thus, by Theorem \ref{nilpotent-connected}, if $n$ is not a prime power, then $D(D_n)$ is connected. On the other hand, if $n$ is a prime power, then by using Theorem \ref{akbari-theorem}, we have $D(D_n)$ to be the empty graph.
\end{proof}

\section{When $G$ is a symmetric group}
In the earlier section, we proved that groups with trivial center has connected difference graphs. In this section, we start with an important family of groups with trivial center, namely $S_n$, the symmetric group on $n$ symbols. The main result of this section is: 
{\theorem \label{Sn-connected-theorem} $D(S_n)$ is connected if and only if $n \geq 8$.}\\

We first give a brief outline of the theorem. We first show that for any vertex $u$ in $D(S_n)$, we get a path joining $u$ to a suitable transposition $\tau_u$, and then show that any two transpositions share a common neighbour in $D(S_n)$. Before proving these results formally, we introduce few notations and prove some lemmas which will be crucial in the proof of the above theorem. Let $\sigma \in S_n$. Define $\{\sigma\}=\{i: \sigma(i)\neq i \}$ and $[\sigma]=|\{\sigma\}|$, i.e., $[\sigma]$ is the number of integers in $\{1,2,\ldots,n\}$ which are not fixed by $\sigma$.


{\lemma \label{2-fixed-points} Let $\sigma \in S_n$ such that there exists $a_1 \in \{1,2,\ldots,n\}$ and $k\in \mathbb{N}$ with $\sigma(a_1)\neq a_1$ and $\sigma^k(a_1)=a_1$. Then there exists $a_2(\neq a_1)$ such that $\sigma^k(a_2)=a_2$.}
\begin{proof}
	Let $\sigma=\sigma_1\cdot \sigma_2\cdots \sigma_l$ be the disjoint cycle decomposition of $\sigma$. Thus $\sigma^k=\sigma^k_1\cdot \sigma^k_2\cdots \sigma^k_l$. Also as $\sigma(a_1)\neq a_1$ and $\sigma^k(a_1)=a_1$, there exists $i$ such that $\sigma_i(a_1)\neq a_1$ and and $\sigma^k_i(a_1)=a_1$.
	
	Let $\sigma_i=(a_1 a_2 \ldots a_t)$ where $t\geq 2$. If $k<t$, then $\sigma^k_i(a_1)=a_{1+k}\neq a_1$, a contradiction. Thus $k\geq t$. By division algorithm, let $k=tq+r$, where $0\leq r <t$. Then $\sigma^k_i=(\sigma^t_i)^q\cdot \sigma^r_i=\sigma^r_i$. If $0<r<t$, then arguing as above $\sigma^k_i(a_1)=\sigma^r_i(a_1)=a_{1+r}\neq a_1$, a contradiction. Thus $r=0$, i.e., $k=tq$ and $\sigma^k_i(a_i)=a_i$, for $i=1,2,\ldots,t$. In particular, $\sigma^k_i(a_2)=a_2$. 
	
	As $\sigma_1, \sigma_2,\ldots \sigma_l$ are disjoint cycles, $\sigma_j(a_2)=a_2$ for $j\neq i$. Thus $\sigma^k(a_2)=a_2$.
\end{proof}

{\lemma \label{prime-power-to-transposition} Let $\sigma \in S_n$ be a vertex in $D(S_n)$, where $n\geq 8$. If $o(\sigma)=p^\alpha$ for some odd prime $p$, then there exists a path joining $\sigma$ and a transposition $(a_1a_2)$ in $D(S_n)$.}
\begin{proof}
{\bf Case 1: ($\sigma$ is a $p^\alpha$-cycle)}	If $[\sigma]\leq n-2$, then there exists $a_1,a_2\in \{1,2,\ldots,n\}$ which are fixed by $\sigma$. Then as $o(\sigma)$ is odd, by Proposition \ref{coprime}, $\sigma\sim (a_1a_2)$ in $D(S_n)$. If $[\sigma]= n$ or $n-1$, then we make the following claim:\\
{\it Claim 1:} If $[\sigma]= n$ or $n-1$, then $\sigma$ does not belong to $D(S_n)$, i.e, $\sigma$ is an isolated vertex of $\mathsf{EPow}(S_n)-\mathsf{Pow}(S_n)$.\\
{\it Proof of Claim 1:}
If possible, let $\sigma$ has a neighbour $\tau_1$ in $D(S_n)$, then $\langle\sigma,\tau_1\rangle$ is a cyclic subgroup, say $\langle x \rangle$, of $S_n$. Thus $\sigma \in \langle x \rangle$. First we consider the case $[\sigma]=n$. So $[x]=[\sigma]=n$. Note that, as $\sigma$ is a $p^\alpha$-cycle, $x$ is also a $p^\alpha$-cycle and $\langle x\rangle$ is a cyclic $p$-subgroup of $S_n$. Hence by Proposition~\ref{p-subgroup}, $\sigma \not\sim \tau_1$ in $D(S_n)$, a contradiction. Let us now consider the case $[\sigma]=n-1$. Thus $[x]=n$ or $n-1$. If $[x]=n$, then $x$ has no fixed points and $\sigma$ being a power of $x$ has $1$ fixed point. So, by Lemma \ref{2-fixed-points}, $\sigma$ has at least two fixed points, i.e., $[\sigma]\leq n-2$, a contradiction. If $[x]=n-1$, then $[\sigma]=[x]=n-1$ and hence $x$ is also a $p^\alpha$-cycle and $\langle x\rangle$ is a cyclic $p$-subgroup of $S_n$. Hence by Proposition~\ref{p-subgroup}, $\sigma \not\sim \tau_1$ in $D(S_n)$, a contradiction. Thus Claim 1 holds.

{\bf Case 2: ($\sigma$ is product of more than one disjoint cycles)} Let $\sigma=\sigma_1\cdot \sigma_2\cdots\sigma_l$. If $[\sigma]\leq n-2$, then there exists $a_1,a_2\in \{1,2,\ldots,n\}$ which are fixed by $\sigma$. Then as $o(\sigma)$ is odd, by Proposition \ref{coprime}, $\sigma\sim (a_1a_2)$ in $D(S_n)$. If $[\sigma]= n$ or $n-1$, arguing as in Claim 1, we can show that $[\sigma]=[x]$ where $\sigma\sim \tau_1$ in $D(S_n)$ and $\langle\sigma,\tau_1\rangle=\langle x\rangle$. Let $x=\rho_1\cdot \rho_2\cdots \rho_k$ and $o(\sigma_i)=p^{\alpha_i}$ for $i=1,2,\ldots,l$.\\
{\it Claim 2:} $l\geq k$.\\
{\it Proof of Claim 2:} As $\sigma$ is a power of $x$, say $x^\beta$, then $\sigma_1\cdot \sigma_2\cdots\sigma_l=\rho^\beta_1\cdot \rho^\beta_2\cdots \rho^\beta_k$. Two possibilities may happen: (i) $\rho^\beta_i=\sigma_j$ for all $i$, (ii)  Some of the $\rho^\beta_i$ may split into more than one cycle, and/or (iii) some of the $\rho^\beta_i$ may vanish, i.e., it may be identity. However, if the third possibility occurs, then $\sigma$ will be having more fixed points than $x$. But, as $[\sigma]=[x]$, this can not happen. Thus either (i) or (ii) occurs. For (i), $l=k$ and for (ii) $l>k$. Hence Claim 2 follows.  

If $l=k$ holds, then $\sigma_1\cdot \sigma_2\cdots\sigma_l=\rho^\beta_1\cdot \rho^\beta_2\cdots \rho^\beta_l$, i.e., $\sigma_i=\rho^\beta_{\pi(i)}$ for all $i$ and for some permutation $\pi \in S_l$. Thus $o(\sigma_i)\leq o(\rho_{\pi(i)})$ for all $i$. As $\sigma_i$'s are $\rho_j$'s are cycles, we have $[\sigma_i]\leq [\rho_{\pi(i)}]$ for all $i$. If $[\sigma_i]< [\rho_{\pi(i)}]$ for some $i$, then $[\sigma]=[\sigma_1]+\cdots+[\sigma_i]+\cdots [\sigma_l]< [\rho_1]+\cdots+ [\rho_l]=[x]$, a contradiction, as $[\sigma]=[x]$. Thus $o(\sigma_i)= o(\rho_{\pi(i)})$ for all $i$ and hence $o(\sigma)=o(x)$. Also, we have $\sigma \in \langle x \rangle$. Thus, it follows that $x \in \langle \sigma \rangle$. Again, as $\tau_1 \in \langle x\rangle$, we have $\tau_1\in \langle \sigma\rangle$. But this implies $\sigma \not\sim \tau_1$ in $D(S_n)$, a contradiction.

Thus $l>k$. We recall that $\sigma=\sigma_1\cdot \sigma_2\cdots\sigma_l=\rho^\beta_1\cdot \rho^\beta_2\cdots \rho^\beta_k=x^\beta$ and $o(\sigma_i)=p^{\alpha_i}$ for $i=1,2,\ldots,l$. If the $\alpha_i$'s are distinct for all $i$, i.e., if each $\sigma_i$ are cycles of distinct length, then no two $\sigma_i,\sigma_j$ with $i\neq j$ is obtained by raising the same $\rho_t$ to the power $\beta$, because cycles raised to some power always yields either a single cycle or more than one cycles of same length. Thus, without of loss of generality, we can assume that $\sigma_i$ is formed by taking the $\beta$-th power of $\rho_i$ for all $i=1,2,\ldots,l$. But this implies $l\leq k$, a contradiction. So, all the $\alpha_i$'s are not distinct. Without loss of generality, let $o(\sigma_1)=o(\sigma_2)=p^{\alpha_{1}}$. Now, two cases may occur:

{\bf Case 2A: ($l\geq 3$)} Let $\sigma=(a_1a_2\cdots a_{p^{\alpha_1}})(b_1b_2\cdots b_{p^{\alpha_1}})\sigma_3\cdots \sigma_l$. Now, as $p$ is odd, there exist $p^{\alpha_i}$-cycles $\sigma'_i$ such that $\sigma_i=(\sigma'_i)^2$ for $i=3,\ldots,l$. Let $$y=(a_1b_1a_2b_2\cdots a_{p^{\alpha_1}}b_{p^{\alpha_1}})\sigma'_3\cdots \sigma'_l.$$ Clearly $y^2=\sigma$. Also note that $y$ is the product of a $2p^{\alpha_{1}}$-cycle and $(l-2)$ many $p^{\alpha_i}$-cycles, and all the cycles are disjoint. As $o(\sigma)=p^\alpha$, therefore $\alpha\geq \alpha_i$ for all $i$. Thus $(\sigma'_i)^{p^\alpha}$ is identity for $i=3,\ldots,l$. Also, as $\gcd(2p^{\alpha_{1}},p^\alpha)=p^{\alpha_{1}}$, $y^{p^{\alpha}}$ is a product of $p^{\alpha_{1}}$ disjoint $2$-cycles and hence $o(y^{p^{\alpha}})=2$. Thus $\sigma$ and $y^{p^{\alpha}}$ are powers of $y$ and their orders are coprime. Hence $\sigma \sim y^{p^{\alpha}}$ in $D(S_n)$. Again, as $\{y^{p^{\alpha}}\}=\{y\}=\{\sigma_1\}+\{\sigma_2\}$ and $\{\sigma_i\}\cap \{y^{p^{\alpha}} \}=\emptyset$ for $i=3,\ldots,l$, there exists $a,b,c \in \{1,2,\ldots,n\}\setminus \{y^{p^{\alpha}}\}$. As $\gcd(o(abc),o(y^{p^{\alpha}}))=\gcd(3,2)=1$ and $(abc)$ commutes with $y^{p^{\alpha}}$, by Proposition \ref{coprime}, we have $y^{p^{\alpha}}\sim (abc)$ in $D(S_n)$. Hence, we get the following path $$\sigma \sim y^{p^{\alpha}}\sim (abc)\sim (a_1a_2)$$

{\bf Case 2B: ($l=2$)} In this case, we have $\sigma=\sigma_1\cdot \sigma_2=(a_1a_2\cdots a_{p^{\alpha}})(b_1b_2\cdots b_{p^{\alpha}})$ and $o(\sigma)=o(\sigma_1)=o(\sigma_2)=p^\alpha$. Recall that $[\sigma]=n$ or $n-1$. Thus either $n=2p^\alpha$ or $n-1=2p^\alpha$. As $n\geq 8$, we have $p^\alpha\geq 4$. Let $y=(a_1b_1a_2b_2\cdots a_{p^{\alpha}}b_{p^{\alpha}})$. Then, as in Case 2A, $y^2=\sigma$ and $y^{p^{\alpha}}$ is the product of $p^\alpha$ disjoint two cycles. Let $\tau=y^{p^{\alpha}}=(c_1c_2)(c_3c_4)(c_5c_6)\tau_4 \cdots \tau_{p^\alpha}$, where $\tau_i$'s are disjoint transpositions. Then $\sigma \sim \tau$ in $D(S_n)$. Let $z=(c_1c_3c_5c_2c_4c_6)\tau_4 \cdots \tau_{p^\alpha}$. Then $z^3=\tau$ and $z^2=(c_1c_5c_4)(c_3c_2c_6)$. Thus $\tau\sim z^2$ in $D(S_n)$. Again, as $n\geq 8$, there exist $a_1,a_2 \in \{1,2,\ldots,n\}\setminus \{z^2\}$. Hence $z^2\sim (a_1a_2)$. Combining we get the path $\sigma\sim \tau\sim z^2\sim (a_1a_2)$.
\end{proof}

Now, we are in a position to prove the main theorem of this section.

\noindent {\bf Proof of Theorem \ref{Sn-connected-theorem}:} Let $\sigma_1,\sigma_2$ be two vertices in $D(S_n)$. Then by Lemma \ref{to-prime-power} and Lemma \ref{prime-power-to-transposition}, there exist transpositions $(a_1b_1)$ and $(a_2b_2)$ such that there exists a path $P_1$ joining $\sigma_1$ and $(a_1b_1)$ and a path $P_2$ joining $\sigma_2$ and $(a_2b_2)$. As $n\geq 8$, therefore there exists a $3$-cycle $(a_3b_3c_3)$ which is disjoint with both $(a_1b_1)$ and $(a_2b_2)$. Thus we have $(a_1b_1)\sim (a_3b_3c_3) \sim (a_2b_2)$, and hence we get a path joining $\sigma_1$ and $\sigma_2$ in $D(S_n)$, and $D(S_n)$ is connected for $n\geq 8$.

On the other hand, it can checked by direct computation that $D(S_n)$ is disconnected for $n=5,6,7$ and $D(S_n)$ is empty for $n\leq 4$.\qed 

\section{Difference graph of Alternating Group}

In the previous section, we concentrated on the difference graph of $S_n$ which have trivial center. In particular, we have seen that for symmetric group $S_n$ with $n \geq 8$, the difference graph $D(S_n)$ is connected. In this section we concentrate on a family of simple groups $A_n$, the alternating group on $n$ symbols. 

For a permutation $\sigma \in S_n$, we write $\sigma$ as $\sigma_1 \cdot \sigma_2 \cdots \sigma_l$ where $\sigma_i$ are cycles of length $>1$ and we say that $\sigma$ is a product of $l$ disjoint nontrivial cycles. We know that the sign of a permutation $\sigma$ can be defined from its decomposition into the product of transpositions as
$sign(\sigma) = (-1)^m$ 
where $m$ is the number of transpositions in the decomposition. 
It can be easily seen that if $\sigma=\sigma_1\cdot \sigma_2 \cdots \sigma_l$ (where $\sigma_i$ are nontrivial cycles), 
the sign of $\sigma$ is $(-1)^{[\sigma]-l}.$ The following lemma is immediate. 

\begin{lemma}
	\label{lem:criterion_sigma_even}
	For a positive integer $n$ and $\sigma=\sigma_1 \cdot \sigma_2 \cdots \sigma_l \in S_n$ where $\sigma_i$ are nontrivial cycles, the permutation $\sigma \in A_n$ if and only if $[\sigma]-l$ is even. 
\end{lemma}

We now prove some other lemmas which will be crucial to prove the main result of this Section.

\begin{lemma}
	\label{lem:sigma_palpha_cycle_An}
	Let $\sigma \in A_n$, where $n\geq 3$ and let $\sigma $ be a $p^\alpha$-cycle for some prime $p\neq 3$. 
	If $[\sigma]= n$ or $n-1$ or $n-2$ , then $\sigma$ does not belong to $D(A_n)$, i.e, $\sigma$ is an isolated vertex of $\mathsf{EPow}(A_n)-\mathsf{Pow}(A_n)$.
\end{lemma}

\begin{proof}
	If $[\sigma]= n$ or $n-1$, the proof is same as for the case of symmetric group.
	So, let $[\sigma]=n-2.$
	If possible, let $\sigma$ has a neighbour $\tau_1$ in $D(A_n)$, then $\langle\sigma,\tau_1\rangle$ is a cyclic subgroup, say $\langle x \rangle$, of $A_n$. Thus $\sigma \in \langle x \rangle$. By Lemma \ref{2-fixed-points}, we cannot have $[x]=n-1.$ 
	
	Suppose $[x]=n.$ As $x^{\beta}=\sigma$ for some positive integer $\beta$ and $\sigma$ is a $(n-2)$-cycle, we must have that $x$ is the product of a $(n-2)$- cycle and a $2$ cycle. As $\sigma \in A_n$, by Lemma \ref{lem:criterion_sigma_even} we have $n-3$ even and thus $x$ must be an odd permutation. Thus $[x]=n$ is not possible. 
	
	Thus, the only possibility is $[x]=n-2.$ 
	Hence $x$ is also a $p^\alpha$-cycle and $\langle x\rangle$ is a cyclic $p$-subgroup of $A_n$. Hence by Proposition~\ref{p-subgroup}, $\sigma \not\sim \tau_1$ in $D(A_n)$, a contradiction. Thus the proof is complete. 
\end{proof}

\begin{lemma}
	\label{lem:sigma_has_3cycles_samelength_or_2cycles_repeated_lengths}
	Let $\sigma \in A_n$ with $[\sigma]=n$ or $n-1$ and $o(\sigma)=p^r$. Suppose $\sigma$ is a product of $l(>1)$ nontrivial cycles, that is, $\sigma=\sigma_1 \cdot \sigma_2 \cdots \sigma_l.$ If $\sigma$ is not isolated, then at least one of the following two conditions must hold: 
	\begin{enumerate}
		\item $\sigma$ has at least $3$ nontrivial cycles of same length. 
		
		\item There exists $1 \leq i_1< i_2< i_3< i_4\leq l$ such that 
		(i) the length of $\sigma_{i_1}$ and $\sigma_{i_2}$ are same and (ii) the length of $\sigma_{i_3}$ and $\sigma_{i_4}$ are same. 
	\end{enumerate} 
\end{lemma} 

\begin{proof}
	let $\sigma$ has a neighbour $\tau_1$ in $D(A_n)$, then $\langle\sigma,\tau_1\rangle$ is a cyclic subgroup, say $\langle x \rangle$, of $A_n$. Thus $\sigma \in \langle x \rangle$. Let $x=\rho_1\cdot \rho_2\cdots \rho_k$ and $o(\sigma_i)=p^{\alpha_i}$ for $i=1,2,\ldots,l$. Proceeding similarly as in the proof of Lemma \ref{prime-power-to-transposition}, we get that 
	all the $\alpha_i$'s are not distinct.
	
	If $\sigma$ does not satisfy any of the $2$ given conditions then there exists $ 1 \leq i_1< i_2 \leq l$ such that (i) $\alpha_{i_1}=\alpha_{i_2}$ , (ii) $\alpha_i \neq \alpha_j$ for any $i,j \in [l]-\{i_1, i_2\}$ and (iii) for any $i \in [l]-\{i_1, i_2\}$, we have  $\alpha_i \neq  \alpha_{i_1}$. That is $\sigma$ has only two nontrivial cycles of same length and all the other nontrivial cycles are of mutually different length. Thus the number of nontrivial cycles of $x$ is exactly $1$ less than the number of nontrivial cycles of $\sigma$ and as $\sigma$ is even and $[\sigma]=[x]$, the permutation $x \notin A_n.$ Thus $\sigma$ must satisfy one of the two aforesaid conditions. This completes the proof. 
\end{proof}

\begin{remark}
	\label{rem:two_2cycles_to_3cycle}
	Let $n \geq 10$. Let $ \pi = (a_1 a_2)(a_3 a_4)$ and $\sigma=(a_5 a_6 a_7)$ where $a_i \neq a_j$ for $i \neq j.$  Then $\pi \sim \sigma$ in $D(A_n).$ Thus, any two $2$ $3$-cycles are connected by a path of length $2$ in $D(A_n).$ Thus, if a permutation $\pi \in D(A_n)$ is connected to a particular $3$-cycle $(a b c)$ by a path of length $t$, then it is connected to any $3$ cycle by a path of length $\leq t+2.$ 
\end{remark}

\begin{lemma}
	\label{lem:6_2cycles_to_3cycle}
	Let $\pi \in A_n$ be a product of $k$ 2-cycles where $n \geq 16$ and $k \geq 6.$ Then there exists $a,b,c,d$ such that there is a path of length $2$ between $\pi$ and $(a b)(c d)$. Thus, there exists $e,f,g$ such that there is a path of length $3$ between $\pi$ and $(e f g)$.
\end{lemma}

\begin{proof}
	Let $\pi=(c_1c_2)(c_3c_4)(c_5c_6)(c_7c_8)(c_9c_{10})(c_{11}c_{12})\tau_7 \cdots \tau_k,$ where $\tau_i$s are disjoint transpositions. Let $$z=(c_1c_3c_5c_2c_4c_6)(c_7c_9c_{11}c_8c_{10}c_{12})\tau_7 \cdots \tau_{p^\alpha}.$$  It is easy to see that as $\pi$ was even, we also have that $z$ is an even permutation. We also have $z^3=\pi$. Further the permutation $z^2$ is a product of $4$ $3$-cycles. Thus $\pi \sim z^2$ in $D(A_n)$. Again, as $n\geq 16$, there exist $a,b,c,d \in \{1,2,\ldots,n\}\setminus \{z^2\}$. Hence $z^2\sim (a b)(cd)$. Combining we get the path $$\pi \sim z^2 \sim (a b)(c d).$$ This proves the Lemma. 
\end{proof}

\begin{lemma}
	\label{lem:4_3cycles_to_3cycle}
	Let $\pi \in A_n$ be a product of $k$ 3-cycles where $n \geq 16$ and $k \geq 4.$ Then there exists $a,b,c$ such that there is a path of length $2$ between $\pi$ and $(a b c)$. 
\end{lemma}

\begin{proof}
	Let $\pi=(c_1c_2c_3)(c_4c_5c_6)(c_7c_8c_9)(c_{10}c_{11}c_{12})\tau_5 \cdots \tau_{k},$ where $\tau_i$s are cycles of length $3$. Let 
	
	$$z= (c_1 c_4 c_2 c_5 c_3 c_6)(c_7 c_{10} c_8 c_{11} c_9 c_{12}) \tau_5' \cdots \tau_{p^\alpha}',$$
	where for $i \geq 5$, $\tau_i'$s are such that $(\tau_i')^2=\tau_i.$
	Note that $z \in A_n$ and $z^2=\tau$. Moreover, we have $z^3=(c_1 c_5)( c_4  c_3) (c_2 c_6) (c_7 c_{11})(c_{10} c_9)(c_8 c_{12}).$  Thus $\pi \sim z^3$ in $D(A_n)$. Again, as $n\geq 16$, there exist $a, b,c \in \{1,2,\ldots,n\}\setminus \{z^3\}$. Hence $z^3 \sim (a b c)$. Combining we get the path $$\pi \sim z^3 \sim (a b c).$$
	The proof is complete. 
\end{proof}

\begin{lemma} \label{odd-prime-power-to-transposition-alternating} Let $\sigma \in A_n$ be a vertex in $D(A_n)$, where $n\geq 18$. If $o(\sigma)=p^\alpha$ for some prime $p\neq 3$, then there exists $a, b, c \in [n]$ such that there is a path joining $\sigma$ and the $3$-cycle $(a b c)$. Thus, there is a path joining $\sigma$ and any $3$-cycle.  
\end{lemma}

\begin{proof}
	If $\sigma$ is a $p^\alpha$-cycle, by Lemma \ref{lem:sigma_palpha_cycle_An}	we must have $[\sigma]\leq n-3$. Thus there exists $a, b ,c \in \{1,2,\ldots,n\}$ which are fixed by $\sigma$. Then as $o(\sigma)$ is not divisible by $3$, by Proposition \ref{coprime}, $\sigma\sim (a_1 a_2 a_3)$ in $D(A_n)$.
	
	So, let $\sigma$ has more than one cycles of length divisible by $p$. Here also, if $[\sigma] \leq n-3$, we are done. So, we assume $[\sigma]=n-2 $ or $n-1$ or $n$.

	{\bf Case 1: $[\sigma]=n:$} Throughout Case 1, we will assume $n \geq 16.$ Let $\tau$ be a neighbor of $\sigma$ in $D(A_n)$. Then $\langle\sigma,\tau_1\rangle$ is a cyclic subgroup, say $\langle x \rangle$, of $A_n$. Thus $\sigma \in \langle x \rangle$. Here we must have $[x]=n.$ Let $\sigma=\sigma_1 \sigma_2 \cdots \sigma_l$. By Lemma \ref{lem:sigma_has_3cycles_samelength_or_2cycles_repeated_lengths}, we have either Case 1A or Case 1B. 
	
	{\bf Case 1A:} $\sigma$ has atleast $3$ nontrivial cycles of same length. Without loss of generality, let $o(\sigma_1)=o(\sigma_2)=o(\sigma_3)=p^{\alpha_{1}}$. Now, two cases may occur:
	
	{	\underline {\it Subcase 1Ai: ($l\geq 4$)} }  We first consider $p$ to be an odd prime. Let $$\sigma=(a_1a_2\cdots a_{p^{\alpha_1}})(b_1b_2\cdots b_{p^{\alpha_1}})(c_1c_2\cdots c_{p^{\alpha_1}})\sigma_4 \cdots \sigma_l.$$
	
	Now, as $p$ is not $3$, there exist $p^{\alpha_i}$-cycles $\sigma'_i$ such that $\sigma_i=(\sigma'_i)^3$ for $i=4,\ldots,l$. Consider the permutation $$y=(a_1b_1c_1a_2b_2c_2\cdots a_{p^{\alpha_1}}b_{p^{\alpha_1}}c_{p^{\alpha_1}})\sigma'_4\cdots \sigma'_l.$$ We can see that $y^3=\sigma$ and the sign of $y$ is same as the sign of $\sigma.$ Therefore, $y \in A_n.$  Also note that $y$ is the product of a $3p^{\alpha_{1}}$-cycle and $(l-3)$ many $p^{\alpha_i}$-cycles, and all the cycles are disjoint. As $o(\sigma)=p^\alpha$, therefore $\alpha\geq \alpha_i$ for all $i$. Thus $(\sigma'_i)^{p^\alpha}$ is identity for $i=4,\ldots,l$. It is clear that $y^{p^{\alpha}}$ is a product of $p^{\alpha_{1}}$ disjoint $3$-cycles and hence $o(y^{p^{\alpha}})=3$. Thus $\sigma$ and $y^{p^{\alpha}}$ are powers of $y$ and their orders are coprime and both are even permutations. Hence $\sigma \sim y^{p^{\alpha}}$ in $D(A_n)$. As $l \geq 4,$ there exists $a,b,c,d \in \{1,2,\ldots,n\}\setminus \{y^{p^{\alpha}}\}$. As $\gcd(o((ab)(cd)),o(y^{p^{\alpha}}))=\gcd(2,3)=1$ and $(ab)(cd)$ commutes with $y^{p^{\alpha}}$, by Proposition \ref{coprime}, we have $y^{p^{\alpha}}\sim (ab)(cd)$ in $D(A_n)$. Hence, we get the following path $$\sigma \sim y^{p^{\alpha}}\sim (ab)(cd).$$
	Now by Remark \ref{rem:two_2cycles_to_3cycle}, we have a path from $y^{p^{\alpha}}$ to $(a b c).$ 
	
	We now consider $p=2$ and we can see that the same proof goes through for $p=2$ unless $l=4$ and the length of $\sigma_4=2.$ In this case, we have $$\sigma= (a_1a_2\cdots a_{2^{\alpha}})(b_1b_2\cdots b_{2^{\alpha}})(c_1c_2\cdots c_{2^{\alpha}})(d_1 d_2) . $$  As $n \geq 16$, we have $ \alpha \geq 3.$ Let $y= (a_1b_1c_1a_2b_2c_2\cdots a_{2^{\alpha}}b_{2^{\alpha}}c_{2^{\alpha}})(d_1 d_2)$. Using Lemma \ref{lem:criterion_sigma_even} and the fact that $\sigma$ is even, we get that $y$ must be even and $y^3=\sigma$. Moreover, $\tau=y^{3.2^{\alpha-1}  } $ is a product of $3.2^{\alpha-1} \geq 6$ many $2$-cycles and using Lemma \ref{lem:6_2cycles_to_3cycle} we are done.

	\underline { {\it Subcase 1Aii: ($l=3$)} }  When $l=3$ we cannot have $p=2$ otherwise $\sigma$ will be forced to become an odd permutation. Thus, we have $$\sigma=\sigma_1\cdot \sigma_2 \cdot \sigma_3=(a_1a_2\cdots a_{p^{\alpha}})(b_1b_2\cdots b_{p^{\alpha}})(c_1c_2\cdots c_{p^{\alpha}})$$ and $o(\sigma)=o(\sigma_1)=o(\sigma_2)= o(\sigma_3)=p^\alpha$ with $p \geq 5.$ As $n\geq 16$, we have $p^\alpha\geq 5$. Let us consider the permutation $y=(a_1b_1c_1a_2b_2c_2\cdots a_{p^{\alpha}}b_{p^{\alpha}}c_{p^{\alpha}})$. We first note that $y$ is an even permutation. Moreover, we have  $y^3=\sigma$ and $y^{p^{\alpha}}$ is the product of $p^\alpha$ disjoint three cycles. Thus, we have $\sigma \sim y^{p^{\alpha}}.$ By Lemma \ref{lem:4_3cycles_to_3cycle}, we have a path from $\sigma$ to $(a b c).$

	
	{\bf Case 1B:} If $\sigma $ satisfy both the conditions of Lemma \ref{lem:sigma_has_3cycles_samelength_or_2cycles_repeated_lengths}, then by Case 1A, we are done. So, let $\sigma$ do not have any $3$ nontrivial cycles of same lentgh and there exists $1 \leq i_1< i_2< i_3< i_4\leq l$ such that 
	(i) the length of $\sigma_{i_1}$ and $\sigma_{i_2}$ are same and (ii) the length of $\sigma_{i_3}$ and $\sigma_{i_4}$ are same. W.l.o.g. let $o(\sigma_1)=o(\sigma_2)=p^{\alpha_{1}}$ and 
	$o(\sigma_3)=o(\sigma_4)=p^{\alpha_{2}}$ and let $\alpha_1 < \alpha_2$. If $p=2$ and none of the cycles of $\sigma $ has been repeated more than twice, then each cycle length of $x$ is also a power of $2$ and thus the order of any neighbor of $\sigma $ is a power of $2$, which contradicts Lemma \ref{to-prime-power}. Hence, in Case 1B, we always take $p$ to be an odd prime. 
	Now, we consider the following $2$ cases. 
	
	\underline{{\it Subcase 1Bi: ($l\geq 5$)}}  Let 
	$$\sigma=(a_1a_2\cdots a_{p^{\alpha_1}})(b_1b_2\cdots b_{p^{\alpha_1}})(c_1c_2\cdots c_{p^{\alpha_2}}) (d_1d_2\cdots d_{p^{\alpha_2}})
	\sigma_5 \cdots \sigma_l.$$ Now, as $p$ is not $2$, there exist $p^{\alpha_i}$-cycles $\sigma'_i$ such that $\sigma_i=(\sigma'_i)^2$ for $i=5,\ldots,l$. Consider the permutation $$y=(a_1b_1a_2b_2\cdots a_{p^{\alpha_1}}b_{p^{\alpha_1}}) (c_1d_1c_2d_2\cdots c_{p^{\alpha_2}}d_{p^{\alpha_2}}) \sigma'_5 \cdots \sigma'_l.$$

	We can see that the sign of the permutation $y$ is same as the sign of $\sigma$ and hence $y \in A_n$. We also have $y^2=\sigma$. As $o(\sigma)=p^\alpha$, therefore $\alpha\geq \alpha_i$ for all $i$. Thus $(\sigma'_i)^{p^\alpha}$ is identity for $i=5,\ldots,l$. It is clear that $y^{p^{\alpha}}$ is a product of $p^{\alpha_{1}}+p^{\alpha_{2}}$ disjoint $2$-cycles and hence $o(y^{p^{\alpha}})=2$. Thus $\sigma$ and $y^{p^{\alpha}}$ are powers of $y$ and their orders are coprime and both are even permutations. Hence $\sigma \sim y^{p^{\alpha}}$ in $D(A_n)$. As $l \geq 5,$ there exists $a,b,c \in \{1,2,\ldots,n\}\setminus \{y^{p^{\alpha}}\}$. As $\gcd(o((abc)),o(y^{p^{\alpha}}))=\gcd(2,3)=1$ and $(abc)$ commutes with $y^{p^{\alpha}}$, by Proposition \ref{coprime}, we have $y^{p^{\alpha}}\sim (abc)$ in $D(A_n)$. Hence, we get the following path $$\sigma \sim y^{p^{\alpha}}\sim (abc).$$
	
	\underline{	{\it Subcase 1Bii: ($l=4$)} }  Let 
	$$\sigma=(a_1a_2\cdots a_{p^{\alpha_1}})(b_1b_2\cdots b_{p^{\alpha_1}})(c_1c_2\cdots c_{p^{\alpha_2}}) (d_1d_2\cdots d_{p^{\alpha_2}}).$$
	Take the permutation
	$$y=(a_1b_1a_2b_2\cdots a_{p^{\alpha_1}}b_{p^{\alpha_1}}) (c_1d_1c_2d_2\cdots c_{p^{\alpha_2}}d_{p^{\alpha_2}}) .$$ Then, $y$ is even and $o(y^{p^{\alpha_2}})=2.$
	Then, as in Subcase 1Bi we note that $y^{p^{\alpha_2}}$ is the product of $p^{\alpha_1}+p^{\alpha_2}$ disjoint two cycles and as $n \geq 16$, we must have $p^{\alpha_1}+p^{\alpha_2}\geq 6$. Using Lemma \ref{lem:6_2cycles_to_3cycle}, we are done. 
	
	Thus, we are done for the case $[\sigma]=n.$ 
	
	
	{\bf Case 2: $[\sigma]=n-1:$} In this case, we will assume $n \geq 17.$ Then $n-1 \geq 16$ and now we can proceed in an identical manner as in Case 1. 
	
	{\bf Case 3: $[\sigma]=n-2:$} Throughout this case, we will assume $n \geq 18$ and thus $n-2 \geq 16.$ 
	%
	%
	%
	Let $\sigma$ has a neighbour $\tau_1$ in $D(A_n)$, then $\langle\sigma,\tau_1\rangle$ is a cyclic subgroup, say $\langle x \rangle$, of $A_n$. Thus $\sigma \in \langle x \rangle$. By Lemma \ref{2-fixed-points}, we cannot have $[x]=n-1.$  So, either $[x]=n-2$, or $[x]=n.$ If $[x]=n-2$  we can proceed in exactly same manner as in Case 1 and there is no problem as $n-2 \geq 16.$ So, let $[x]=n.$ Let $[n]-\{\sigma\}=\{d_1, d_2\}.$
	Without loss of generality, let $o(\sigma_1)=o(\sigma_2)=p^{\alpha_{1}}$. Now, two cases may occur:
	
	{\bf Case 3A:} Here we consider $p$ to be an odd prime. 
	
	\underline { {\it Subcase 3Ai: ($l\geq 3$)} } Let $\sigma=(a_1a_2\cdots a_{p^{\alpha_1}})(b_1b_2\cdots b_{p^{\alpha_1}})\sigma_3\cdots \sigma_l$. Now, as $p$ is odd, there exist $p^{\alpha_i}$-cycles $\sigma'_i$ such that $\sigma_i=(\sigma'_i)^2$ for $i=3,\ldots,l$. Let $$y=(a_1b_1a_2b_2\cdots a_{p^{\alpha_1}}b_{p^{\alpha_1}})(d_1 d_2)\sigma'_3\cdots \sigma'_l.$$ We first convince ourself that $y$ is even. Clearly $y^2=\sigma$. Also note that $y$ is the product of one $2p^{\alpha_{1}}$-cycle, one $2$-cycle and $(l-2)$ many $p^{\alpha_i}$-cycles, and all the cycles are disjoint. As $o(\sigma)=p^\alpha$, therefore $\alpha\geq \alpha_i$ for all $i$. Thus $(\sigma'_i)^{p^\alpha}$ is identity for $i=3,\ldots,l$. Also, as $\gcd(2p^{\alpha_{1}},p^\alpha)=p^{\alpha_{1}}$, $y^{p^{\alpha}}$ is a product of $p^{\alpha_{1}}+1$ disjoint $2$-cycles and hence $o(y^{p^{\alpha}})=2$. Thus $\sigma$ and $y^{p^{\alpha}}$ are powers of $y$ and their orders are coprime. Hence $\sigma \sim y^{p^{\alpha}}$ in $D(A_n)$. Again, as $n \geq 18,$ there exists $a,b,c \in \{1,2,\ldots,n\}\setminus \{y^{p^{\alpha}}\}$. As $\gcd(o(abc),o(y^{p^{\alpha}}))=\gcd(3,2)=1$ and $(abc)$ commutes with $y^{p^{\alpha}}$, by Proposition \ref{coprime}, we have $y^{p^{\alpha}}\sim (abc)$ in $D(A_n)$. Hence, we get the following path $$\sigma \sim y^{p^{\alpha}}\sim (abc). $$
	
	\underline{{\it Subcase3Aii: ($l=2$)}}  In this case, we have $\sigma=\sigma_1\cdot \sigma_2=(a_1a_2\cdots a_{p^{\alpha}})(b_1b_2\cdots b_{p^{\alpha}})$ and $o(\sigma)=o(\sigma_1)=o(\sigma_2)=p^\alpha$.  As $n\geq 18$, we have $p^\alpha\geq 6$. Let $y=(a_1b_1a_2b_2\cdots a_{p^{\alpha}}b_{p^{\alpha}})(d_1 d_2)$. Then, as in Case 3A, $y^2=\sigma$ and $y^{p^{\alpha}}$ is the product of $p^\alpha+1$ disjoint two cycles. 
	Let $$\tau=y^{p^{\alpha}}=(c_1 c_2) (c_3 c_4) (c_5 c_6) (c_7 c_8) (c_9 c_{10}) (c_{11} c_{12}) \tau_7 \cdots \tau_{p^{\alpha}+1}$$ where $\tau_i$ are transpositions. Again, by using Lemma \ref{lem:6_2cycles_to_3cycle}, we have a path from $\sigma$ to a $3$-cycle $(a b c).$
	
	{\bf Case 3B:} We now consider $p=2$. We first observe that if none of the cycles of $\sigma$ repeats more than twice,  then each cycle length of $x$ is also a power of $2$ and thus the order of any neighbor of $\sigma $ is a power of $2$, which contradicts Lemma \ref{to-prime-power}. Hence, 
	$\sigma$ has at least $3$ nontrivial cycles of same length.
	Moreover, $\sigma$ must have at least $4$ nontrivial cycles otherwise $\sigma$ will be odd. So, let
	
	$$\sigma=(a_1a_2\cdots a_{2^{\alpha_1}})(b_1b_2\cdots b_{2^{\alpha_1}}) (c_1c_2\cdots c_{2^{\alpha_1}}) \sigma_4 \sigma_5 \cdots \sigma_l.$$ Let us take $$y=(a_1b_1 c_1 a_2b_2c_2\cdots a_{2^{\alpha_1}}b_{2^{\alpha_1}}c_{2^{\alpha_1}}) \sigma_4'\sigma'_5\cdots \sigma'_l$$ where $(\sigma_i')^3= \sigma_i$ for $i \geq 4.$ Then $y^{2^{\alpha} }$ is a product of $2^{\alpha_1}$ disjoint $3$ cycles. Moreover as $\sigma$ is even, we see that $y$ is even and hence $y^{2^{\alpha}}$ is also even. Thus, we have $\sigma \sim y^{2^{\alpha}}$ in $D(A_n).$ If $\alpha_1 \leq 2$, then we have $[n]-\{ y^{2^{\alpha_1}} \} \geq 4$ and thus we have $a,b,c,d$ such that $y^{2^{\alpha_1}} \sim (a b)(c d) .$ If $\alpha_1 \geq 3$ then  $y^{2^{\alpha_1}}$ is a product of atleast $8$ disjoint $3$-cycles and we can again use Lemma \ref{lem:4_3cycles_to_3cycle} to get a path from $\sigma$ to a $3$-cycle $(a b c).$  
	
	This completes the proof. 
\end{proof} 
We are now in a position to prove the main result of this section. 

{\theorem \label{An-connected-theorem} If $n \geq 18$, then $D(A_n)$ is connected.}

\begin{proof} Let $\sigma_1,\sigma_2$ be two vertices in $D(A_n)$. Then by Lemma \ref{to-prime-power}, there exists an element, say $\alpha_1$ of prime power order say $q_1^{\beta_1} (q_1 \neq 3)$ such that $\sigma_1$ is connected by a path to $\alpha_1.$ Similarly, there exists an element, say $\alpha_2$ of prime power order say $q_2^{\beta_2} (q_2 \neq 3)$ such that $\sigma_2$ is connected by a path to $\alpha_2.$
	Using Lemma \ref{odd-prime-power-to-transposition-alternating} we have a path joining $\alpha_1$ to a $3$-cycle $(a b c) $ and a path joining $\alpha_2$ to the same $3$-cycle $(a b c)$.
	The proof is complete. 
\end{proof}

\section{Groups with Trivial Center}
It has been observed that if $G$ is a finite group with non-trivial center, then $D(G)$ is connected. So, the natural question is to enquire about the connectedness of $D(G)$ when $Z(G)$ is trivial. It was shown that $D(S_n)$ is connected if and only if $n\geq 8$. Thus, it is not possible to have a general answer to this question. So, we investigate the connectedness of some special families of groups with trivial center.

{\theorem \label{direct-product-theorem} Let $G_1,G_2, G_3$ be three finite groups such that $G_1\times G_2 \times G_3$ is not a $p$-group, then $D(G_1\times G_2 \times G_3)$ is connected.}
\begin{proof}
	Let $a=(a_1,a_2,a_3), b=(b_1,b_2,b_3)$ be two vertices in $D(G_1\times G_2 \times G_3)$. Thus, by Lemma \ref{to-prime-power}, there exist $a'=(a'_1,a'_2,a'_3), b'=(b'_1,b'_2,b'_3) \in D(G_1\times G_2 \times G_3)$ such that $a\sim a'$ with $o(a')=p^\alpha$ and $b \sim b'$ with $o(b')=q^\beta$ for some primes $p$ and $q$.
	
	{\bf Case 1:} ($p\neq q$) Since $o(a')=p^\alpha$, therefore one of $o(a'_1),o(a'_2),o(a'_3)$ is $p^\alpha$. Without loss of generality, let $o(a'_1)=p^\alpha$. 
	
	\noindent {\it Claim 1:} $(a_1,a_2,a_3)\sim (a'_1,e_2,e_3)$ in $D(G_1\times G_2 \times G_3)$.\\
	{\it Proof of Claim:} As $a \sim a'$ in $D(G_1\times G_2 \times G_3)$, we have $\langle a,a'\rangle=\langle (x_1,x_2,x_3)\rangle$. Then $\langle a_1,a'_1 \rangle \subseteq \langle x_1 \rangle$ and hence $\langle (a_1,a_2,a_3), (a'_1,e_2,e_3) \rangle \subseteq \langle (x_1,a_2,a_3) \rangle $. Since $o((a'_1,a'_2,a'_3))=o((a'_1,e_2,e_3))=p^\alpha$, none of $(a_1,a_2,a_3)$ and $(a'_1,e_2,e_3)$ is a power of the other. Hence $(a_1,a_2,a_3)\sim (a'_1,e_2,e_3)$ in $D(G_1\times G_2 \times G_3)$
	
	Again, one of $o(b'_1),o(b'_2),o(b'_3)$ is $q^\beta$. If $o(b'_2)=q^\beta$ or $o(b'_3)=q^\beta$, say $o(b'_2)=q^\beta$, then by Claim 1, we have $(a_1,a_2,a_3)\sim (a'_1,e_2,e_3)$ and $(e_1,b'_2,e_3) \sim (b_1,b_2,b_3)$ in $D(G_1\times G_2 \times G_3)$. Again, as $o((a'_1,e_2,e_3))=p^\alpha$ and $o((e_1,b'_2,e_3))=q^\beta$ and they commute, we have $(a'_1,e_2,e_3)\sim (e_1,b'_2,e_3)$. Thus, we get a path: $$(a_1,a_2,a_3)\sim (a'_1,e_2,e_3)\sim (e_1,b'_2,e_3) \sim (b_1,b_2,b_3)\mbox{ in }D(G_1\times G_2 \times G_3).$$
	
	If none of $o(b'_2)$ or $o(b'_3)$ is $q^\beta$, then we must have $o(b'_1)=q^\beta$. Then by Claim 1, $(b_1,b_2,b_3) \sim (b'_1,e_2,e_3)$.
	
	{\bf Case 1A:} (There exists a prime $r$ other than $p$ and $q$ such that $r$ divides $|G_2|$ or $|G_3|$) Without loss of generality, let $r$ divides $|G_2|$. Then we get an element $c_2 \in G_2$ with $o(c_2)=r$. Hence we get the following path joining $a$ and $b$:
	$$(a_1,a_2,a_3)\sim (a'_1,e_2,e_3)\sim (e_1,c_2,e_3)\sim (b'_1,e_2,e_3)\sim (b_1,b_2,b_3)$$
	
	{\bf Case 1B:} ($|G_2|$ and $|G_3|$ has no prime factors other than $p$ and $q$) Without loss of generality, let $q$ divides $|G_2|$. Then there exists $c_2 \in G_2$ with $o(c_2)=q$. Thus we get a path $(a_1,a_2,a_3)\sim (a'_1,e_2,e_3)\sim (e_1,c_2,e_3)$. As $(b'_1,e_2,e_3)$ is not isolated (using Claim 1) and $o((b'_1,e_2,e_3))=q^\beta$, by Lemma \ref{to-prime-power}, there exists a prime $r(\neq q)$ and $(d_1,d_2,d_3) \in G_1\times G_2 \times G_3$ such that $\o(d_1,d_2,d_3))=r^\gamma$ and $(b'_1,e_2,e_3)\sim (d_1,d_2,d_3)$ in $D(G_1\times G_2 \times G_3)$.
	
	If $o(d_1)$ or $o(d_3)$ is $r^\gamma$, say $o(d_1)=r^\gamma$, then by repeated use of Claim 1, we get the following path in $D(G_1\times G_2 \times G_3)$:
	$$(a_1,a_2,a_3)\sim (a'_1,e_2,e_3)\sim (e_1,c_2,e_3)\sim (d_1,e_2,e_3)\sim (b'_1,e_2,e_3)\sim (b_1,b_2,b_3).$$
	
	If $o(d_1),o(d_3)\neq r^\gamma$, we must have $o(d_2)=r^\gamma$. If $r\neq p$, then we get the following path in $D(G_1\times G_2 \times G_3)$:
	$$(a_1,a_2,a_3)\sim (a'_1,e_2,e_3)\sim (e_1,d_2,e_3)\sim (b'_1,e_2,e_3)\sim (b_1,b_2,b_3).$$
	
	So, we are left with the case when $r=p$. Then we consider $|G_3|$. As $p$ and $q$ are the only possible distinct prime factors of $|G_3|$, $G_3$ must contain an element $f_3$ of order $p$ or $q$. 
	\begin{itemize}
		\item If $o(f_3)=p$, then we get the following path in $D(G_1\times G_2 \times G_3)$: $$(a_1,a_2,a_3)\sim (a'_1,e_2,e_3)\sim (e_1,c_2,e_3)\sim (e_1,e_2,f_3)\sim (b'_1,e_2,e_3) \sim (b_1,b_2,b_3).$$
		\item If $o(f_3)=q$, then we get the following path in $D(G_1\times G_2 \times G_3)$: $$(a_1,a_2,a_3)\sim (a'_1,e_2,e_3)\sim (e_1,e_2,f_3)\sim (e_1,d_2,e_3)\sim (b'_1,e_2,e_3) \sim (b_1,b_2,b_3).$$
	\end{itemize}
Thus, summing up all the cases, we have shown that $D(G_1\times G_2 \times G_3)$ is connected if $p\neq q$, i.e., {\bf Case 1} is resolved. Now, we turn towards the case when $p=q$.

{\bf Case 2:} ($p=q$) If $p=q$, as $G_1\times G_2 \times G_3$ is not a $p$-group, there exists a prime $r\neq p$, which divides $|G_1\times G_2 \times G_3|$, i.e., $r$ divides one of $|G_1|$, $|G_2|$ and $|G_3|$. Proceeding similarly as in Case 1, it can be shown that $D(G_1\times G_2 \times G_3)$ is connected.
\end{proof}

{\remark In light of Theorem \ref{direct-product-theorem}, if $G_i$'s are groups with trivial centers, then the difference graph of their direct product is connected. So, Theorem \ref{direct-product-theorem} gives us a natural way to construct infinitely many finite groups $G$ with trivial center such that $D(G)$ is connected.}
	
{\remark It is to be noted that Theorem \ref{direct-product-theorem} can be generalized to direct product of $n$ groups where $n\geq 3$ in a similar way. However, we add an word of caution that Theorem \ref{direct-product-theorem} may not be true for direct product of two groups. For example, $D(S_3\times S_3)$ is a disconnected graph of order $10$ with $2$ isomorphic components of order $5$ each.}

In the next theorem, we identify some families of graphs which can be expressed as direct product of two groups $G_1$ and $G_2$ such that $Z(G_1)$, $Z(G_2)$ are trivial but $D(G_1 \times G_2)$ is connected.

{\theorem \label{direct-product-dihedral} $D(D_n \times D_m)$ is disconnected if and only if $n$ and $m$ are powers of same odd prime.}
\begin{proof}
	If any of $n$ or $m$ is even, then $Z(D_n\times D_m)$ is non-trivial and as a result $D(D_n \times D_m)$ is connected. Note that if both $n$ and $m$ are powers of $2$, i.e., $D_n\times D_m$ is a $2$-group, then $D(D_n \times D_m)$ is empty. Thus we deal only with the case when both $m$ and $n$ are odd.
	
	Let $D_n=\langle r_1,s_1: r^n_1=s^2_1=e, s_1r_1s_1=r^{-1}_1 \rangle$ and $D_m=\langle r_2,s_2: r^m_2=s^2_2=e, s_2r_2s_2=r^{-1}_2 \rangle$. The idea of the proof is to show that 
	\begin{itemize}
		\item any vertex $(a,b)\in D(D_n \times D_m)$ is joined by a path to either an element of the form $(r^x_1,e)$ or $(e,r^y_2)$ in $D(D_n \times D_m)$.
		\item There exists a path joining $(r^x_1,e)$ and $(r^y_1,e)$ in $D(D_n \times D_m)$.
		\item There exists a path joining $(r^x_1,e)$ and $(e,r^y_2)$ in $D(D_n \times D_m)$.	
	\end{itemize}

{\it Claim 1:} Any element the form $(r^i_1s_1,r^j_2s_2)$ is not a vertex of $D(D_n \times D_m)$, i.e., it is an isolated vertex in the difference of enhanced power graph and power graph of $D_n \times D_m$.

{\it Proof of Claim 1:}	We show that $(r^i_1s_1,r^j_2s_2)$ is not adjacent to any vertex in enhanced power graph of $D_n \times D_m$. If not, it is adjacent to some vertex in the enhanced power graph, which in turn implies that $(r^i_1s_1,r^j_2s_2)$ belongs to some cyclic group containing $(r^i_1s_1,r^j_2s_2)$ and atleast another element of $D_n \times D_m$. However, the only cyclic group containing $(r^i_1s_1,r^j_2s_2)$ is $\langle(r^i_1s_1,r^j_2s_2)\rangle$ which contains no non-identity element other than $(r^i_1s_1,r^j_2s_2)$.

Thus from Claim 1, it follows that any vertex $(a,b)$ of $D(D_n \times D_m)$ is one of the three forms: $(r^i_1, r^j_2s_2), (r^i_1s_1, r^j_2)$ or $(r^i_1, r^j_2)$.

{\bf Case 1:} Let $(a,b)=(r^i_1, r^j_2s_2) \in D(D_n \times D_m)$. By Lemma \ref{to-prime-power}, $(r^i_1, r^j_2s_2)$ is adjacent to some $(a_1,b_1)$ of order $p^\alpha$ in $D(D_n \times D_m)$. Also, note that $(r^i_1, r^j_2s_2)$ belongs only in cyclic groups of the form $\langle (r^*_1, r^j_2s_2) \rangle$, i.e., $(a_1,b_1)\in \langle (r^*_1, r^j_2s_2) \rangle$. As $o(r^j_2s_2)=2$, if $p$ is odd, then $b_1=e$ and we get $(a,b)\sim (r^*_1, e)$. If $p=2$, then $o((a_1,b_1))=2^\alpha$ and again by Lemma \ref{to-prime-power}, $(a_1,b_1) \sim (a_2,b_2)$ where $o((a_2,b_2))=q^\beta$, $q$ being an odd prime. By similar argument and using the fact that $o(r^j_2s_2)=2$, we get $b_2=e$ and $(a,b)\sim (a_1,b_1)\sim (r^*_1, e)$.

{\bf Case 2:} Let $(a,b)=(r^i_1s_1, r^j_2) \in D(D_n \times D_m)$. We proceed similarly as in Case 1 to get a path joining $(a,b)$ and a vertex of the form $(e,r^*_2)$.

{\bf Case 3:} Let $(a,b)=(r^i_1, r^j_2) \in D(D_n \times D_m)$. By Lemma \ref{to-prime-power}, there exists an element $(a_1,b_1) \in D(D_n \times D_m)$ such that $(a,b)\sim (a_1,b_1)$ and $o((a_1,b_1))=p^\alpha$. As $n$ and $m$ are not powers of same prime, there exists a prime $q(\neq p)$ such that $p|n$ or $p|m$. Let $p|n$. Then there exists an element $r^x_1 \in D_n$ of order $p$ and we get the following path in $D(D_n \times D_m)$: $(a,b)\sim (a_1,b_1)\sim (r^x_1,e)$, where the adjacency of the last two vertices follows from Proposition \ref{coprime}.

Thus combining the above three cases, we have shown that any vertex $(a,b)\in D(D_n \times D_m)$ is joined by a path to either an element of the form $(r^x_1,e)$ or $(e,r^y_2)$ in $D(D_n \times D_m)$.

As $n$ and $m$ are odd, it is clear that two vertices of the form $(r^x_1,e)$ and $(r^y_1,e)$ has a common neighbour $(e,s_2)$. 

Now, we construct a path between two vertices of the form $(r^x_1,e)$ and $(e,r^y_2)$ in $D(D_n \times D_m)$ as follows: $$(r^x_1,e)\sim (e,s_2)\sim (r^z_1,e)\sim (e,r^w_2)\sim (s_1,e)\sim (e,r^y_2),$$ where we choose $w$ and $z$ such that $o(r^z_1)$ and $o(r^w_2)$ are powers of different primes. This can be done as $n$ and $m$ are not powers of same odd prime. Thus, if $n$ and $m$ are not powers of same odd prime, $D(D_n \times D_m)$ is connected.

Now, we consider the case, when $n$ and $m$ are powers of same odd prime, say $n=p^\alpha$ and $m=p^\beta$. We will show that there is no path joining $(r_1,e)$ and $(e,r_2)$ in  $D(D_n \times D_m)$. We start by noting that $(r_1,e)$ and $(e,r_2)$ are indeed vertices of $D(D_n \times D_m)$, as $(r_1,e)\sim (e,s_2)$ and $(e,r_2)\sim (s_1,e)$ in $D(D_n \times D_m)$, i.e., $(r_1,e)$ and $(e,r_2)$ are not isolated vertices of the difference of enhanced power graph and power graph of $D_n \times D_m$.

We observe that $(r_1,e)$ belongs to cyclic subgroups of one of the following two types: $\langle (r^i_1,r^j_2) \rangle$ or $\langle (r^i_1,r^j_2s_2) \rangle$. However, as $\langle (r^i_1,r^j_2) \rangle$ are $p$-groups, the subgraph induced by these subgroups is empty and hence Lemma \ref{subgroup-induced}, elements in these subgroups do not contribute to a path joining $(r_1,e)$ to $(e,r_2)$. Thus, neighbours of $(r_1,e)$ lies in subgroups of the type $\langle (r^i_1,r^j_2s_2) \rangle$. The elements of $\langle (r^i_1,r^j_2s_2) \rangle$ are either of the form $(r^*_1,e)$ or $(r^*_1,r^j_2s_2)$. As $(r_1,e)\not\sim (r^*_1,e)$, thus $(r_1,e)$ is adjacent only to vertices of the form $(r^*_1,r^j_2s_2)$. Similarly, $(r^*_1,r^j_2s_2)$ are only adjacent to vertices of the form $(r^{**}_1,e)$.

Proceeding similarly, $(e,r^*_2)$ and $(r^j_1s_1,r^{**}_2)$ are exclusive neighbours. Thus there does not exist any path joining $(r_1,e)$ and $(e,r_2)$ in  $D(D_n \times D_m)$ and hence $D(D_n \times D_m)$ is disconnected.
\end{proof}
	 
\section{The clique number of $D(G)$}

If $S$ is a set of elements of a group $G$ such that every two elements of $S$
generate a cyclic group, then $S$ generates a cyclic group (see
\cite[Lemma 32]{akbari-cameron}). Hence every maximal clique in
$\mathsf{EPow}(G)$ is
a maximal cyclic subgroup; thus every maximal clique in $D(G)$ is contained in
a maximal cyclic subgroup of $G$. So we first need to find the clique number
of $D(\mathbb{Z}_n)$ for an integer $n$.

\begin{proposition}
The clique number of $D(\mathbb{Z}_n)$ is equal to the maximum size of an
antichain in the lattice of divisors of $n$.
\end{proposition}

\begin{proof}
Two elements of $\mathbb{Z}_n$ which are joined in $D(\mathbb{Z}_n)$ must have
different orders, and neither divides the other; conversely, two elements with
this property are joined in $D(\mathbb{Z}_n)$.
\end{proof}

The maximum size of an antichain in the lattice of divisors of $n$ was found
by de Bruijn \emph{et al.} \cite{debruijn}; this is a generalization of
Sperner's lemma. Define the \emph{degree} of $n$ to be the number of prime
divisors of $n$, counted with multiplicity. Let $m$ be the degree of $n$.
Then an antichain of maximal size consists of all the divisors of $n$ of
degree $m/2$, if $m$ is even; or either all divisors of degree $(m-1)/2$ or
all divisors of degree $(m+1)/2$, if $m$ is odd.

For example, a clique of maximal size in $C_{360}$ is obtained by choosing
elements of orders $8$, $12$, $18$, $20$, $30$, $45$.

\begin{proposition}
The clique number of $D(G)$ is equal to the maximum clique number of a cyclic
subgroup of $G$, so is determined by the set of orders of elements of $G$.
\end{proposition}

This is now clear from our earlier remarks.

\begin{corollary}
The graph $D(G)$ is triangle-free if and only if the order of every element of
$G$ is a prime power of the form $p^kq$ where $p,q$ are primes.
\end{corollary}

\begin{proof}
Let $g$ be an element of $G$ of order $n$. If $n$ is not of the forms in the
Corollary, either it has three prime divisors $p,q,r$, or it is divisible by
$p^2q^2$ for some prime $q$. In the first case, elements of orders $p,q,r$
in $\langle g\rangle$ form a triangle; in the second, elements of orders
$p^2,pq,q^2$ form a triangle.
\end{proof}

We note that the chromatic number of $D(G)$ may be larger than the clique
number. This example is taken from \cite{cp}. In the symmetric group $S_8$,
the orders of elements are $1$, $2$, $3$, $4$, $5$, $6$, $7$, $8$, $10$, $12$
and $15$. By the Corollary above, the clique number of $D(S_8)$ is $2$. But
$D(S_8)$ is not bipartite, since it contains a $5$-cycle
\[\{(1,2), (3,4,5), (6,7), (1,2,3), (4,5,6,7,8)\}.\]

\section{Perfectness and other properties}

It is known that the power graph of a finite group is perfect (see \cite{power-perfect}). On the other hand, the question ``For which groups is the enhanced power graph perfect?'' is still unresolved. In this section, we discuss perfectness of $D(G)$. We also say something about the related problem of when $D(G)$ is a
cograph.

\subsection{Graph classes, induced subgraphs and twin reduction}

The \emph{clique number} of a graph is the size of the largest complete
subgraph, and the \emph{chromatic number} is the smallest number of colours
required to colour the vertices so that adjacent vertices are given different
colours. The clique number does not exceed the chromatic number since, in a
proper colouring, all vertices of a clique are given different colours. 
A graph $\Gamma$ is \emph{perfect} if every induced subgraph of $\Gamma$ has
clique number equal to chromatic number. The \emph{strong perfect graph
theorem}, conjectured by Berge and proved by
Chudnovsky \emph{et al.}~\cite{PGT}, states that a graph is perfect if and
only if it does not contain either an odd cycle or the complement of an odd
cycle as an induced subgraph. It follows that a graph is perfect if and only
if its complement is perfect. This statement, known as the \emph{weak perfect
graph theorem}, was proved earlier by Lov\'asz. A number of graph classes
are known to be perfect, including bipartite graphs and comparability graphs
of partial orders.

Several other classes of graphs also have characterizations in terms of
forbidden induced subgraphs. Among these, we will only consider the class
of \emph{cographs}, graphs which contain no induced subgraph isomorphic to the
$4$-vertex path $P_4$. Since $P_4$ is isomorphic to its complement and any
cycle of length greater than $4$ contains an induced $P_4$, we see that cographs
are perfect. Cographs form the smallest class of
graphs which can be built from the $1$-vertex graph by the operations of
complementation and disjoint union.

Any class of graphs defined by forbidden induced subgraphs is subgraph-closed.
We will use two tools to investigate when difference graphs are perfect or
belong to one of the other classes:
\begin{enumerate}
\item If $H$ is a subgroup of $G$, then the induced subgraph of $D(G)$ on
the set $H$, after removing isolated vertices, is $D(H)$. So the class
of groups for which the difference graph belongs to one of the above classes
is subgroup-closed.
\item Two vertices of a graph are \emph{twins} if they have the same
neighbours (possibly excluding each other). The process of \emph{twin
reduction} involves finding a pair of twins and identifying them, and
continuing until no further twins remain. The result of twin reduction is
(up to isomorphism) independent of the process of reduction, and is called
the \emph{cokernel} of the graph (since $\Gamma$ is a cograph if and only if
its cokernel is the $1$-vertex graph). The cokernel is an induced subgraph
of the original graph and (in the cases we consider) is often much smaller
and more amenable to analysis. The important fact is given in the next result.
\end{enumerate}

\begin{proposition}
Let $\mathcal{F}$ be a class of finite graphs, and suppose that no graph in
$\mathcal{F}$ possesses a pair of twin vertices. Then a graph $\Gamma$ has
no induced subgraph in $\mathcal{F}$ if and only if the same applies to the
cokernel of $\Gamma$.
\end{proposition}

This result applies to perfect graphs and to cographs. 

\subsection{Perfect difference graphs}

\begin{proposition}
The difference graph of a cyclic group is perfect.
\label{cyclic-perfect}
\end{proposition}

\begin{proof}
The enhanced power graph of a cyclic group is complete, so the difference
graph is simply the complement of the power graph, which as we saw is
perfect. Now we can invoke the Weak Perfect Graph Theorem.
\end{proof}

Moreover, we can determine the cyclic groups for which the difference graph
is a cograph: for the class of cographs is self-complementary, and the
nilpotent groups whose power graph is a cograph were determined in
\cite[Theorem 12]{MCM}. The result is:

\begin{proposition}
The difference graph of a cyclic group $\mathbb{Z}_n$ is a cograph if and
only if either $n$ is a prime power or $n$ is the product of two distinct
primes.
\end{proposition}

This result does not extend to abelian groups. We saw that the difference
graphs of abelian groups, even those which are direct products of two
isomorphic cyclic groups, are universal, and in particular, we can find one
which embeds a $5$-cycle.

\begin{theorem}\label{pqr-perfect}
Let $G$ be a group of order $pq$, $p^2q$, $p^3q$, $p^2q^2$, or $pqr$,
where $p,q,r$ are distinct primes. Then $D(G)$ is perfect.
\end{theorem}

\begin{proof}
By Theorem~\ref{cyclic-perfect}, we may assume that $G$ is not cyclic. Also, we
may assume that $G$ is not an EPPO group (one with all elements of prime power
order), since for such a group $G$ the graph $D(G)$ has no edges.

Let $G$ be a group of order $pq$. Then either $G$ is cyclic or it is an EPPO
group, and the result follows.

Let $G$ be a group of order $p^2q$. We may assume that $G$ is non-cyclic but
has elements of order $pq$.
We claim that elements of order $pq$ are isolated in $D(G)$.
If not, let $x$ be an element of order $pq$ which is adjacent to a vertex $y$. Then from the adjacency condition of difference graph, $o(y)=p^2$. But this implies $\langle x,y\rangle$ is a cyclic group of order $p^2q$ in $G$, a contradiction.
Thus any edge in $E(G)$ must join elements of orders $p$ and $q$. So $E(G)$ is
bipartite (with the sets of elements of these orders as bipartite sets) and
hence perfect.

Next let $G$ be a group of order $p^3q$, and suppose that $G$ is not cyclic.
The possible orders of elements of $G$ are $p$, $p^2$, $p^3$, $q$, $pq$, or
$p^2q$. Elements of order $p^2q$ cannot be adjacent to elements of order 
dividing $p^2q$, or to elements of order $p^3$; so they are isolated. Thus
any edge of $E(G)$ must join a vertex of order a power of $p$ with one of 
order $q$ or $pq$; so the graph is bipartite, and hence perfect.

Now let $G$ be a non-cyclic group of order $p^2q^2$, where $p>q$. There is no
element of order $p^2q^2$, and arguing as above we see that elements of orders
$p^2q$ or $pq^2$ are isolated, and elements of orders $p^2$ and $q^2$ cannot
be adjacent. We can assume that $G$ is not an EPPO group; so it contains
elements of order $pq$. Moreover, we can assume there are elements of orders
$p^2$ and $q^2$. For, if there are no elements of order $q^2$, then all edges
join elements with order divisible by $q$ to elements with order a power of $p$, and the graph is bipartite, and hence perfect. Hence the Sylow subgroups of $G$
are cyclic. Now there is a normal $q$-complement, which is cyclic of order
$p^2$ (unless $p=3$ and $q=2$). Now an element of order $q$ or $q^2$ which acts nontrivially on a cyclic
group of order $p^2$ must have trivial centralizer there. So either the group
$G$ is cyclic, or there is no element of order $pq^2$. So elements of order
$q^2$ are isolated, and $D(G)$ is bipartite by the same argument as before.

Finally, let $G$ be a non-cyclic group of order $pqr$.
As in the previous cases, we can show that elements of order $pq, pr$ and $qr$ are isolated in $D(G)$. (If an $o(x)=pq$ and $x$ is joined to $y$, then if
$r\mid o(y)$ then $\langle x,y\rangle$ is cyclic of order $pqr$, while if
$r\nmid o(y)$ then $y$ is a power of $x$.) So all edges join elements of
distinct prime orders. If at most two of $p,q,r$ occur as orders of elements
then the graph is bipartite, and hence perfect; so all three occur. Now if
$p>q>r$, then $G$ contains a normal Sylow $p$-subgroup $P$, and $P$ commutes
with elements of orders $q$ and $r$, so is central in $G$; thus $G$ is the
direct product of $\mathbb{Z}_p$ with a group $H$ of order $qr$, which by
assumption is cyclic. So $G$ is cyclic, a contradiction.

In the case when $p=3$ and $q=2$, $G$ is of order $36$. There are $14$ non-isomorphic groups of order $36$ and it can be checked using \textsf{GAP}~\cite{GAP4} that all of them yield perfect difference graphs. This completes the proof.
\end{proof}

\subsection{Imperfect difference graphs}

Now we show that some groups have imperfect difference graphs.

\begin{proposition}
\begin{enumerate}
\item For any three distinct primes $p$, $q$, $r$, the group
$\mathbb{Z}_{pqr}\times\mathbb{Z}_p$ has imperfect difference graph.
\item For two distinct odd primes $p$ and $q$, the group
$Q_8\times\mathbb{Z}_{pq}$, where $Q_8$ is the quaternion group of order~$8$,
has imperfect difference graph.
\end{enumerate}
\end{proposition}

(Note that all proper subgroups of these groups have perfect difference graphs.)

\begin{proof}
(a) Let $a,b,c$ be elements of $\mathbb{Z}_{pqr}$ with orders $p$, $q$, $r$
respectively, and $a'$ an element of $\mathbb{Z}_p$. Then the set
\[\{(c,e),(b,e),(ac,e),(bc,e),(b,a')\}\]
induces a $5$-cycle in $D(G)$.
\smallskip
(b) Let $a,a'$ be non-commuting elements of order $4$ in $Q_8$, and $b$ and
$c$ elements of orders $p$ and $q$. Then the set
\[\{(e,c),(e,b),(a,c),(e,bc),(a',b)\}\]
induces a $5$-cycle in $D(G)$.
\end{proof}

Now we can determine which nilpotent groups have perfect difference graph. Let
$\pi(G)$ be the number of distinct prime divisors of $G$. We may assume that
$\pi(G)>1$.

\begin{theorem}\label{nilpotent-perfect-theorem}
Let $G$ be a finite nilpotent group.
\begin{enumerate}
\item If $\pi(G)\geq 3$, then $D(G)$ is perfect if and only if $G$ is cyclic.
\item If $\pi(G)=2$, then $D(G)$ is a comparability graph, and hence perfect.
\end{enumerate}
\end{theorem}

\begin{proof}
Suppose that $\pi(G)\geq 3$. If $G$ is cyclic then the theorem follows from Theorem \ref{cyclic-perfect}. If $G$ has a generalized quaternion Sylow
$2$-subgroup, then it contains a subgroup $Q_8\times\mathbb{Z}_{pq}$ for odd
primes $p$ and $q$. Otherwise, by a theorem of Burnside, at least one Sylow
subgroup contains commuting elements of prime order (say $p$), and $G$ contains a subgroup $\mathbb{Z}_{pqr}\times\mathbb{Z}_p$. In either of the last two
cases, $D(G)$ is imperfect, by the preceding Proposition.

If $\pi(G)=2$, then $G\cong H \times K$ where $H$ is the Sylow $p$-subgroup and $K$ is the Sylow $q$-subgroup of $G$. Thus any element of $g \in G$ can be uniquely expressed as $ab$, where $a \in H$ and $b \in K$. Also, order of any element in $G$ is either a power of $p$ or a power of $q$ or product of powers of $p$ and $q$. In the first two cases, we call it a element or vertex of type-I and the last case is denoted by type-II.

       Define a relation $\rightarrow$ on $V(D(G))$ as follows:
        \begin{itemize}
                \item If $x_1,x_2$ are of type-I, then $x_1\rightarrow x_2$ if $\circ(x_1)$ is a power of $p$ and $\circ(x_2)$ is a power of $q$.
                \item If $x_1=a_1b_1,x_2=a_2b_2$ are of type-II, then $x_1\rightarrow x_2$ if $\langle a_2 \rangle \leq \langle a_1 \rangle$ and $\langle b_1 \rangle \leq \langle b_2 \rangle$ and at least one of the inequality is strict.
                \item If $x_1$ is of type-I with $o(x_1)$ is a power of $p$ and $x_2=a_2b_2$ is of type-II, then $x_1\rightarrow x_2$ if $\langle a_2 \rangle < \langle a_1 \rangle$.
                \item If $x_1$ is of type-I with $\circ(x_1)$ is a power of $q$ and $x_2=a_2b_2$ is of type-II, then $x_2\rightarrow x_1$ if $\langle b_2 \rangle < \langle b_1 \rangle$.
        \end{itemize}
    It is easy to check that $\rightarrow$ is anti-symmetric and transitive on $V(D(G))$. The comparability graph of $\rightarrow$ on $V(D(G))$ is given by $x \sim y$ if and only if $x\rightarrow y$ or $y \rightarrow x$. It can be checked that $D(G)$ coincides with the comparability graph of $\rightarrow$ on $V(D(G))$. Hence the theorem follows.
\end{proof}

By arguments similar to those already given, the following results can be shown.
Suppose that $q$ and $r$ are primes with $r\mid q-1$. Then $\mathbb{Z}_q$ has
an automorphism $\alpha$ of order $r$. We define an action $\varphi$ of
$\mathbb{Z}_{r^2}$ on $\mathbb{Z}_q$ where the generator of $\mathbb{Z}_{r^2}$
induces the automorphism $\alpha$. We say that a group $G$ is \emph{minimal
imperfect} if its difference graph is imperfect but, for all proper subgroups
$H$ of $G$, $D(H)$ is perfect.

\begin{theorem}
\begin{enumerate}
\item Let $p,q,r$ be three distinct primes such that $r|q-1$. Then $\mathbb{Z}_p \times (\mathbb{Z}_q \rtimes_\varphi \mathbb{Z}_{r^2})$, where $\varphi$ is defined as above, is a minimal imperfect group.
\item Let $q,r$ be two distinct primes such that $r\mid q-1$. Then $\mathbb{Z}_{q^2} \times (\mathbb{Z}_q \rtimes_\varphi \mathbb{Z}_{r^2})$, where $\varphi$ is defined as above, is a minimal imperfect group.
\end{enumerate}
\end{theorem}

Now we turn to finite simple groups, and first show:

\begin{theorem}
The symmetric group $S_8$, the alternating group $A_9$, and the Janko group
$J_1$ all have imperfect difference graphs.
\end{theorem}

\begin{proof}
In the first two cases we can give an explicit induced $5$-cycle:
\begin{itemize}
\item in $S_8$, the set $\{(1,2), (3,4,5), (6,7), (1,2,3), (4,5,6,7,8)\}$ 
induces a $5$-cycle;
\item in $A_9$, the set
\[\{(1,2,3),(4,5)(6,7),(8,9,1),(2,3)(4,5),(6,7,8),(9,1)(2,3),(4,5,6,7,8)\}\]
induces a $5$-cycle.
\end{itemize}
For $J_1$ we proceed as follows. All information we require is given in the
$\mathbb{ATLAS}$ of Finite Groups~\cite{ATLAS}.

The order of the group is $2^3.3.5.7.11.19$,
and the Sylow $2$-subgroup is elementary abelian.  Elements of orders $7$, $11$
and $19$ commute only with their powers, so are isolated, and deleted in the
difference graph; so all the vertices have orders $2$, $3$ and $5$, and 
vertices which are joined must have different orders. So the graph is tripartite
(that is, has a $3$-colouring). Moreover, there is no element of order $30$, so
the clique number is $2$. To show it is not perfect, we just have to show that
it is not bipartite.

The group contains a subgroup $D_3\times D_5$, and so has a path of length~$3$
joining two commuting involutions, with the elements having orders $2,3,5,2$.
Now take a subgroup isomorphic to the Klein group, and join its elements by
such paths. This produces a closed walk of length $9$, so indeed the graph is
non-bipartite.
\end{proof}

\subsection{Simple groups}

We have some partial results on the question ``Which finite simple groups
have perfect difference graphs?'' According to the preceding section, any
simple group which contains the symmetric group $S_8$, the alternating group
$A_9$ or the Janko group $J_1$ has imperfect difference graph. Among the
sporadic groups, this list includes the Fischer groups, the Baby Monster and
the Monster, the Harada--Norton group, the Conway group $Co_1$, the Thompson
group, the Lyons group, the Higman--Sims group, and the O'Nan group. 
Moreover, of course, the alternating groups $A_n$ for $n\ge9$, and the groups
of Lie type of rank at least $9$ also contain $S_8$ or $A_9$ and so have
imperfect difference graph.

On the other hand, we have:

\begin{theorem}
Let $G$ be the simple group $\mathrm{PSL}(2,q)$ (for prime power $q\ge4$) or
$\mathrm{Sz}(q)$ (for $q$ an odd power of $2$). Then $G$ has perfect difference
graph.
\end{theorem}

\begin{proof}
The simplest cases to deal with are $\mathrm{PSL}(2,q)$ for $q$ a power of $2$
and $\mathrm{Sz}(q)$. These groups have the property that the centralizer of
any element is either cyclic or a $2$-group; and, moreover, distinct 
centralizers meet only in the identity. So the difference graph consists of
isolated vertices together with a disjoint union of difference graphs of cyclic
groups, and so (by Theorem~ \ref{cyclic-perfect}) it is perfect.

Indeed, we also see that for these groups, the difference graph is a cograph
(or a threshold or split graph) if and only if the difference graphs of all
the cyclic subgroups are. So, if $q$ is a power of $2$, then
\begin{itemize}
\item $D(\mathrm{PSL}(2,q))$ is a cograph if and only if each of $q-1$ and
$q+1$ is a prime power or the product of two distinct primes;
\item For $q$ an odd power of~$2$, $D(\mathrm{Sz}(q))$ is a cograph if and only
if each of $q-1$, $q+\sqrt{2q}+1$ and $q-\sqrt{2q}+1$ is either a prime
power or the product of two primes.
\end{itemize}

Consider $\mathrm{PSL}(2,q)$ with $q$ odd. In this group, if elements
$x$ and $y$ have different centralizers, then they are not adjacent in the
difference graph; for, if they commute, they must both be involutions. Thus it
is again true that $D(G)$ is the disjoint union of isolated vertices and the
difference graphs of element centralizers (which are cyclic, dihedral, or
elementary abelian). Moreover, we have a similar characterization of the case
where the difference graph is a cograph: both $(q+1)/2$ and $(q-1)/2$ must be
prime powers or products of two distinct primes.
\end{proof}

The result does not extend to all groups of Lie type of rank $1$. For
example, let $G$ be the Ree group ${}^2G_2(q)=R_1(q)$, where $q$ is an odd
power of $3$. Then $(q+1)/2$ is twice an odd number. Suppose that $(q+1)/2$
has two distinct prime divisors $p$ and $r$. The centralizer of an involution
$t$ in $G$ is isomorphic to $\mathbb{Z}_2\times\mathrm{PSL}(2,q)$, and so
contains a subgroup $\mathbb{Z}_2\times\mathbb{Z}_{2pr}$, whose difference
graph is not perfect, by Theorem~\ref{pqr-perfect}. (It can be shown that, if
$(q+1)/2$ is twice an odd prime power, then $D(R_1(q))$ is perfect; but we
do not give the argument here.)

Further, the difference graphs of the Ree groups are never cographs. For
let $t$ be an involution in $G=R_1(q)$, so that
$C_G(t)\cong\langle t\rangle\times\mathrm{PSL}(2,q)$. As noted above,
$q$ is an odd power of $3$, and $(q-1)/2$ (the order of a cyclic subgroup
of $\mathrm{PSL}(2,q)$) is twice an odd number. Let $p$ be an odd prime
dividing $(q+1)/2$, and let $u$ and $v$ be elements of order $p$ and $2$ in
a cyclic group of order $(q+1)/2$; let $s$ be an element of order $3$ in
$\mathrm{PSL}(2,q)$. Then $\{s,t,u,v\}$ induces a path of length $3$.

This can be seen another way. The (non-simple) smallest Ree group $R_1(3)$
is isomorphic to $\mathrm{P}\Gamma\mathrm{L}(2,8)$, and is contained in all
other Ree groups. Computation shows that the cokernel of its difference graph
is non-trivial. It has $147$ vertices and is bipartite, with bipartite sets
of sizes $63$ and $84$; it has diameter $6$ and girth $10$.

We have not examined the fourth type of rank~$1$ group, the unitary groups
$\mathrm{PSU}(3,q)$, except to note that computation shows that their difference
graphs are perfect for $q=3,4,5$.

Turning to groups of Lie type of rank greater than~$1$, we have the following:

\begin{proposition}
Let $q$ be a prime power, and assume that $q-1$ has at least three distinct
prime divisors. Then $\mathrm{PSL}(3,q)$ has imperfect difference graph.
\end{proposition}

\begin{proof}
Let $F$ be the field of $q$ elements. The subgroup of diagonal matrices in
$\mathrm{SL}(3,q)$ is isomorphic to $F^\times\times F^\times$, under the
map
\[\begin{pmatrix}a&&\\&b&\\&&c\\\end{pmatrix}\mapsto (a,b)\]
(since $abc=1$). This is also a subgroup
of $\mathrm{PSL}(3,q)$ if $3\nmid q-1$, whereas if $3\mid q-1$ then we take
the quotient by the cyclic group of order $3$. With our hypothesis, in either
case we have a subgroup $\mathbb{Z}_{plr}\times\mathbb{Z}_p$, where $p,l,r$ are
distinct primes. The result now follows from Theorem~\ref{pqr-perfect}.
\end{proof}

We note that the argument shows also that, under the same hypothesis on $q$,
$\mathrm{SL}(3,q)$ has imperfect
difference graph. Now many groups of Lie type contain either $\mathrm{PSL}(3,q)$
or $\mathrm{SL}(3,q)$ as a subgroup; in particular, all those of rank at
least~$3$, as we may see by looking at the Coxeter--Dynkin diagrams. In
addition, the group $G_2(q)$ contains $\mathrm{SL}(3,q)$. So for these values
of $q$, the difference graphs of these groups are imperfect.

\subsection{Computational results}

We examined small simple groups computationally using \textsf{GAP}~\cite{GAP4}
and its share package \texttt{GRAPE}~\cite{grape}. We first constructed the
difference graph of the simple group. Then we deleted isolated vertices and
performed twin reduction, leading to some great reductions in number of
vertices.

For groups of order less than that of $J_1$, all had perfect difference graphs.
Indeed, we found the following.
\begin{itemize}
\item \textbf{Simple groups $G$ for which $D(G)$ has no edges:} These are the
simple EPPO groups: $\mathrm{PSL}(2,q)$ for $q=4,7,8,9,17$,
$\mathrm{Sz}(q)$ for $q=8,32$, and $\mathrm{PSL}(3,4)$.
\item \textbf{Simple groups $G$ for which $D(G)$ has edges but is a cograph,
so that the cokernel has a single vertex:} Some further
$\mathrm{PSL}(2,q)$ and $\mathrm{Sz}(q)$, depending on number-theoretic
properties of $q$ (e.g. in our range $\mathrm{PSL}(2,q)$ for $q=11, 13, 16$).
\item \textbf{Simple groups for which $D(G)$ is not a cograph but its cokernel
is bipartite:} some further $\mathrm{PSL}(2,q)$ (e.g. $q=23, 25$),
$\mathrm{PSL}(3,3)$, $\mathrm{PSU}(3,3)$, $M_{11}$, $A_8$, $\mathrm{PSU}(4,2)$,
$\mathrm{PSU}(3,4)$, $M_{12}$, $\mathrm{PSU}(3,5)$.
\end{itemize}

Some of the cokernels in the above list turn out to be very interesting graphs.
For example,
\begin{itemize}
\item \textbf{The Mathieu group $M_{11}$:}
In this case, removal of isolated vertices and twin reduction brings the
number of vertices down from $7920$ to $385$. The resulting graph is bipartite,
with bipartite sets of sizes $165$ and $220$, and the vertices in the two
sets have valencies $4$ and $3$ respectively. The graph has diameter $10$ and
girth $10$; the girth is rather large for a graph if this size. The automorphism
group of the graph is just $M_{11}$.
\item\textbf{The group $\mathrm{PSL}(3,3)$:}
In this case, we found a very natural graph which has not been studied, as far
as we are aware. The vertices are the ordered pairs
$(P,L)$, where $P$ is a point and $L$ a line of the projective plane of 
order $3$ (so $169$ vertices). The pairs fall into two types, \emph{flags}
($P$ incident with $L$) and \emph{antiflags} ($P$ not indident with $L$).
The graph is bipartite: each edge joins a flag to an antiflag. The rule for
adjacency is as follows: the flag $(P,L)$ is incident with the antiflag
$(Q,M)$ if $Q\in L$ and $P\in M$. The automorphism of the graph is
$\mathrm{Aut}(\mathrm{PSL}(3,3)$.
\end{itemize}


\section{Conclusion and Open Issues}
In this paper, we studied the difference graph $D(G)$ of a finite group $G$. The study was mainly based on connectedness and perfectness of such graphs. Some of the problems which arise from this work can be interesting topics of further research.

For a finite group $G$ with non-trivial center, it was shown that $D(G)$ is connected and with diameter less or equal to $6$. However for groups $G$ with trivial center, $D(G)$ may or may not be connected. So the question arises:
\begin{question}
If $G$ has trivial center and $D(G)$ is connected, can $\diam(D(G))$ be greater than $6$? 
\end{question}

\begin{question}
Complete the classification of finite groups whose difference graph is
perfect.
\end{question}

\begin{question}
Find necessary and sufficient conditions on a complete graph with edges
coloured red, green and blue for it to be embeddable in a finite group $G$
such that
\begin{enumerate}
\item red edges are adjacent in the power graph of $G$;
\item green edges are adjacent in the difference graph (that is, in the
enhanced power graph but not in the power graph); and
\item blue edges are non-adjacent in the enhanced power graph.
\end{enumerate}
Necessary conditions are that the red edges form the comparability graph of
a partial order, and if $x$ and $y$ are joined by a green edge then there is
a point $z$ joined to both by green edges. Are these conditions sufficient?
(A similar result for the enhanced power graph and commuting graph was proved
in \cite{survey}.)
\end{question}

\section*{Acknowledgement}
The first author is supported by the PhD fellowship of CSIR (File no. $08/155$ $(0086)/2020-EMR-I$), Govt. of India. The second author acknowledges the Isaac Newton Institute for Mathematical Sciences, Cambridge, for support and hospitality during the programme \textit{Groups, representations and applications: new perspectives}
(supported by \mbox{EPSRC} grant no.\ EP/R014604/1), where he held a Simons
Fellowship. The third author acknowledges the funding of DST grant $SR/FST/MS-I/2019/41$, Govt. of India. The fourth author acknowledges SERB-National Post-Doctoral Fellowship (File No. PDF/2021/001899) during the preparation of this work and profusely thanks Science and Engineering Research Board for this funding.

\medskip

\end{document}